\documentclass[12pt,reqno]{amsart}
\usepackage{amsfonts}
\usepackage{hyperref}
\usepackage{latexsym}
\usepackage{amssymb}
\usepackage{amsmath}
\usepackage{color}
\usepackage{bbm}
\usepackage{tikz}
\usepackage{enumerate}
\usepackage{mathrsfs}
\usepackage{todonotes}
\usepackage{tikz}
\usepackage{pgfplots}

\usepackage[left=2.5cm, top=2.5cm,bottom=2.5cm,right=2.5cm]{geometry}

\newcommand{\Ham}{\mathbb H}
\newcommand{\R}{\mathbb R}
\newcommand{\N}{\mathbb N}
\newcommand{\C}{\mathbb C}
\newcommand{\E}{\mathbb E}
\newcommand{\Pro}{\mathbb P}

\newcommand{\vol}{\mathrm{vol}}
\newcommand{\Mat}{\mathrm{Mat}}

\newcommand{\dint}{{\rm d}}
\newcommand{\Sph}{\ensuremath{{\mathbb S}}}
\newcommand{\B}{\ensuremath{{\mathbb B}}}
\newcommand{\SSS}{\ensuremath{{\mathbb S}}}

\newcommand{\eps}{\varepsilon}
\newcommand{\toweak}{\overset{w}{\underset{n\to\infty}\longrightarrow}}
\newcommand{\eqdistr}{\stackrel{d}{=}}
\newcommand{\ind}{\mathbbm{1}} 

\DeclareMathOperator{\Tr}{Tr}

\renewcommand{\Re}{\operatorname{Re}}  

\newtheorem{thm}{Theorem}[section]
\newtheorem{cor}[thm]{Corollary}
\newtheorem{lemma}[thm]{Lemma}
\newtheorem{df}[thm]{Definition}
\newtheorem{proposition}[thm]{Proposition}

{
\theoremstyle{definition}

\newtheorem{rmk}[thm]{Remark}

}

\def\cM{\mathcal{M}}

\usepackage{nomencl}
\makenomenclature

\allowdisplaybreaks

\begin{document}



\title[Sanov-type large deviations in Schatten classes]{Sanov-type large deviations in Schatten classes}

\author[Z. Kabluchko]{Zakhar Kabluchko}
\address{Zakhar Kabluchko: Institut f\"ur Mathematische Stochastik, Westf\"alische Wilhelms-Uni\-ver\-sit\"at M\"unster, Germany}
\email{zakhar.kabluchko@uni-muenster.de}

\author[J. Prochno]{Joscha Prochno}
\address{Joscha Prochno: Institut f\"ur Mathematik und Wissenschaftliches Rechnen, Karl-Franzens-Universit\"at Graz, Austria} \email{j.prochno@hull.ac.uk}

\author[C. Th\"ale]{Christoph Th\"ale}
\address{Christoph Th\"ale: Fakult\"at f\"ur Mathematik, Ruhr-Universit\"at Bochum, Germany} \email{christoph.thaele@rub.de}




\begin{abstract}
Denote by $\lambda_1(A), \ldots, \lambda_n(A)$ the eigenvalues of an $(n\times n)$-matrix $A$. Let  $Z_n$ be an $(n\times n)$-matrix chosen uniformly at random from the matrix analogue to the classical $\ell_ p^n$-ball, defined as the set of all self-adjoint $(n\times n)$-matrices satisfying $\sum_{k=1}^n |\lambda_k(A)|^p\leq 1$.  We prove a large deviations principle for the (random) spectral measure of the matrix $n^{1/p} Z_n$. As a consequence, we obtain that the spectral measure of $n^{1/p} Z_n$ converges weakly almost surely to a non-random limiting measure given by the Ullman distribution, as $n\to\infty$. The corresponding results for random matrices in Schatten trace classes, where eigenvalues are replaced by the singular values, are also presented. 
\end{abstract}

\maketitle

\tableofcontents

\section{Introduction and main results}

\subsection{General introduction}

The systematic study of high-dimensional convex bodies is one of the central aspects in theory of Asymptotic Geometric Analysis. Probably the most prominent example are the unit balls of the classical $\ell_p^n$-spaces. Their geometry as well as their analytic and probabilistic aspects are well understood by now and we refer the reader to the research monographs and surveys \cite{AsymptoticGeometricAnalysisBookPart1, IsotropicConvexBodies,G2014b, G2014a} and the references therein.

In the local theory of Banach spaces and today in Asymptotic Geometric Analysis there has ever since been a particular interest and focus on the non-commutative settings. These are represented on the one hand by the Schatten trace classes $S_p$ ($0 < p \leq +\infty$), consisting of all compact linear operators on a Hilbert space for which the sequence of their singular values belongs to the sequence space $\ell_p^n$, and on the other hand by their self-adjoint subclasses, where the sequence of eigenvalues values belongs to $\ell_p^n$, which in dimension $n$ are formed by the so-called classical matrix ensembles  $\mathscr H_n(\R)$ (Gaussian orthogonal ensemble = GOE), $\mathscr H_n(\C)$ (Gaussian unitary ensemble = GUE) and $\mathscr H_n(\mathbb H)$ (Gaussian symplectic ensemble = GSE). For example, it were Gordon and Lewis \cite{GL1974} who obtained that the space $\mathcal S_1$ does not have local unconditional structure, Tomzcak-Jaegermann \cite{TJ1974} demonstrated that this space (naturally identified with the projective tensor product $\ell_2\otimes_{\pi}\ell_2$) has Rademacher cotype $2$, Szarek and Tomczak-Jaegermann \cite{ST1980} provided bounds for the volume ratio of $\mathcal S_1^n$, and K\"onig, Meyer and Pajor \cite{KMP1998} proved the boundedness of the isotropic constants of $\mathcal S_p^n$ ($1\leq p \leq +\infty$). More recently, Gu\'edon and Paouris \cite{GP2007} have established concentration of mass properties for the unit balls of Schatten $p$-classes and classical matrix ensembles, Barthe and Cordero-Erausquin \cite{BartheCordero-Erausquin} studied variance estimates, Ch\'avez-Dom\'inguez and Kutzarova determined the Gelfand widths of certain identity mappings between finite-dimensional trace classes $\mathcal S_p$, Radke and Vritsiou \cite{RV2016} and Vritsiou \cite{VritsiouVariance} proved the thin-shell conjecture and the variance conjecture for the operator norm, respectively, Hinrichs, Prochno and Vyb\'iral \cite{HPV17} computed the entropy numbers for natural embeddings of $\mathcal S_p^n$ in all possible regimes, and Kabluchko, Prochno and Th\"ale \cite{KPT2018b, KPT2018a} obtained the precise asymptotic volumes of the unit balls in classical matrix ensembles and Schatten classes, studied volumes of intersections (in the spirit of Schechtman and Schmuckenschl\"ager \cite{SchechtmanSchmuckenschlaeger}) and determined the exact asymptotic volume ratios for $\mathcal S_p^n$ ($0<p\leq +\infty$). It can be seen from all this work referenced above that while those matrix spaces often show a certain similarity to the commutative setting of classical $\ell_p^n$-spaces, there is a considerable difference in the behavior of certain quantities related to the geometry of Banach spaces. In fact, often other methods and tools are needed and proofs can be considerably more involved.
 
It was only recently that the probabilistic concept of a large deviations principle (LDP) was considered in Asymptotic Geometric Analysis by Gantert, Kim and Ramanan \cite{GKR}. Contrary to a central limit theorem the LDP allows one to access non-universal features and unveil properties that distinguish between different convex bodies. In the setting of finite-dimensional $\ell_p^n$-spaces, the authors proved an LDP for 1-dimensional projections of random vectors drawn uniformly from the unit ball $\B_p^n$ of $\ell_p^n$, demonstrating stark changes in large deviation behavior as the parameter $p$ varies. This result was extended by Alonso-Guti\'errez, Prochno and Th\"ale \cite{APT2018} to a higher-dimensional setting in the case where projections to random subspaces are considered, showing that the Euclidean norm of the projection of a random vector uniformly distributed in $\B_p^n$ onto a random subspace satisfies an LDP (see also \cite{APT2017} and \cite{KPT17CCM} for complementing results and other distributions). In his recent PhD thesis, Kim \cite{Kim2017} was able to extend further the results from \cite{APT2018} and \cite{GKR} to more general classes of random vectors satisfying an asymptotic thin-shell-type condition in the spirit of \cite{ABP2003} (see \cite[Assumption 5.1.2]{Kim2017}). Among others, this condition is satisfied by random vectors chosen uniformly at random from a (generalized) Orlicz ball. This body of research is complemented by \cite{KimRamanan}, in which Kim and Ramanan obtained a so-called Sanov-type large deviations principle for the empirical measure of an $n^{1/p}$ multiple of a point drawn from an $\ell_p^n$-sphere with respect to the cone or surface measure. The rate function is essentially shown to be the so-called relative entropy perturbed by some $p$-th moment penalty (see \cite[Equation (3.4)]{KimRamanan}). More precisely (also to allow comparison with the rate function we obtain for non-commutative $\ell_p^n$-spaces), they showed that if $X^{(n,p)}$ is a random vector distributed according to the cone measure $\mu_p$ on an $\ell_p^n$-sphere ($1\leq p \leq +\infty$), then the corresponding sequence of empirical measures satisfies an LDP (on the space of probability measures on $\R$ equipped with the $q$-Wasserstein topology for $q<p$) with good rate function given by
\[
\mathscr J(\nu) = 
\begin{cases}
H(\nu||\mu_p)+ \frac{1}{p}\big(1-m_p(\nu)\big) &\,: m_p(\nu) \leq 1 \\
+\infty &\,: m_p(\nu) >1\,, 
\end{cases}
\]
where $\nu$ is a probability measure on $\R$, $m_p(\nu) = \int_{\R}|x|^p \,\nu(\dint x)$, and 
\[
H(\nu||\mu) = 
\begin{cases}
\int_{\R} \log f(x)\,\nu(\dint x ) & \,: \nu \ll \mu,\, f=\frac{\dint\nu}{\dint\mu} \\
 +\infty &\,: \text{ otherwise}
\end{cases}
\]
is the relative entropy of $\nu$ with respect to the probability measure $\mu$. 

The purpose of the present paper is to leave the geometric setting of classical $\ell_p^n$-spaces and study principles of large deviations in the non-commutative framework of self-adjoint and non self-adjoint Schatten $p$-classes. A large deviations principle for the law of the spectral measure of a Gaussian Wigner matrix has already been obtained by Ben Arous and Guionnet \cite[Theorem 1.1 and Theorem 1.3]{ben_arous_guionnet}. A more general large deviations theorem for random measures (including the case of the empirical eigenvalue distribution of Wishart matrices) has been obtained by Hiai and Petz in \cite[Theorem 1]{hiai_petz_wishart} (see also \cite{hiai_petz}), which followed their preceding ideas from \cite{HP2000}, where the empirical eigenvalue distribution of suitably distributed random unitary matrices was shown to satisfies a large deviations principle as the matrix size goes to infinity. In the same spirit, we shall prove in this paper Sanov-type large deviations principles for the spectral measure of $n^{1/p}$ multiples of random matrices chosen uniformly (or with respect to the cone measure on the boundary) from the unit balls of self-adjoint and non self-adjoint Schatten $p$-classes where $0< p \leq +\infty$ (see Theorems \ref{theo:main} and \ref{theo:main_p_infty} for the self-adjoint case as well as Theorem \ref{thm:MainNonSelfAdjoint} for the non self-adjoint case). In the proofs, we roughly follow a classical strategy in large deviations theory (see, e.g., \cite{AGZ2010, ben_arous_guionnet, hiai_petz_wishart, HP2000}). However, in our case, we need to control the deviations of the empirical measures \textit{and}, in addition, their $p$-th moments towards arbitrary small balls in the product topology of the weak topology on the space of probability measures and the standard topology on $\R$ in the self-adjoint or $\R_+$ in the non self-adjoint set-up, respectively, and then prove exponential tightness. We shall also use a probabilistic representation for random points in the unit balls of classical matrix ensembles obtained recently in \cite{KPT2018a} (see \eqref{eq:schechtman_zinn_uniform} and \eqref{eq:schechtman_zinn_cone}) and a non self-adjoint counterpart (see Proposition \ref{prop:schechtman-zinn-non-self-adjoint}). As we shall see, the good rate function governing the LDPs is essentially given by the logarithmic energy (which is remarkably the same as the negative of Voiculescu's free entropy introduced in \cite{V1993}) and, which is quite interesting, a perturbation by a constant, which is strongly connected to the famous Ullman distribution. In fact, this constant already appeared in our recent works \cite{KPT2018b,KPT2018a}, where the precise asymptotic volume of unit balls in classical matrix ensembles and Schatten trace classes were computed using ideas from the theory of logarithmic potentials with external fields. As a consequence of our LDPs, we obtain a law of large numbers and show that the spectral measure converges weakly almost surely to the Ullman distribution, as the dimension tends to infinity (see Corollary \ref{cor:LLN}).

\subsection{Results for the self-adjoint case}

In order to present our main results for the self-adjoint case in detail, let us introduce some notation. Consider $\beta\in\{1,2,4\}$ and let $\mathscr H_n(\mathbb{F}_\beta)$ be the collection of all self-adjoint $(n\times n)$-matrices with entries from the skew field $\mathbb{F}_\beta$, where $\mathbb{F}_1=\R$, $\mathbb{F}_2=\C$ or $\mathbb{F}_4=\Ham$, the set of Hamiltonian quaternions. The standard Gaussian distribution on $\mathscr H_n(\mathbb{F}_\beta)$ is known as the GOE (Gaussian orthogonal ensemble) if $\beta=1$, the GUE (Gaussian unitary ensemble) if $\beta=2$, and the GSE (Gaussian symplectic ensemble) if $\beta=4$. By $\lambda_1(A),\ldots,\lambda_n(A)$ we denote the (real) eigenvalues of a matrix $A$ from $\mathscr H_n(\mathbb{F}_\beta)$ and consider the following Schatten-type unit ball, which can be regarded as a matrix analogue to the classical $\ell_p^n$-balls and is defined as
$$
\B_{p,\beta}^n := \left\{A\in \mathscr H_n(\mathbb{F}_\beta):\sum_{j=1}^n|\lambda_j(A)|^p \leq 1\right\},\qquad \beta \in\{1,2,4 \}\quad\text{and}\quad  0 < p \leq +\infty.
$$
If $p=+\infty$, then the sum above is replaced by $\max\{|\lambda_j(A)|:j=1,\ldots,n\} $ . The boundary of the matrix ball $\B_{p,\beta}^n$ is denoted by
$$
\Sph_{p,\beta}^{n-1} := \partial \B_{p,\beta}^n  = \left\{A\in \mathscr H_n(\mathbb{F}_\beta):\sum_{j=1}^n|\lambda_j(A)|^p = 1\right\}
$$
with the same convention if $p=+\infty$.
The space $\mathscr H_n(\mathbb{F}_\beta)$ admits a natural scalar product $\langle A, B\rangle = \Re \Tr (A B^*)$ so that it becomes a Euclidean space. The corresponding Riemannian volume on $\mathscr H_n(\mathbb{F}_\beta)$ is denoted by $\text{vol}_{\beta,n}$ and this measure coincides with the suitably normalized $({\beta n(n-1)\over 2} + \beta n)$-dimensional Hausdorff measure on $\mathscr H_n(\mathbb{F}_\beta)$ as follows directly from the area-coarea formula. We can therefore define the uniform distribution on $\B_{p,\beta}^n$. The cone probability measure on $\Sph_{p,\beta}^{n-1}$ is defined as follows: the cone measure of a Borel set $K\subseteq \Sph_{p,\beta}^{n-1}$ is
$$
\frac{\text{vol}_{\beta,n} (\cup_{\lambda\in [0,1]} \lambda K)}{\text{vol}_{\beta,n}(\B_{p,\beta}^n)}.
$$

The main result of this manuscript for the self-adjoint case is the following Sanov-type large deviations principle for random matrices distributed according to the uniform distribution on  $\B_{p,\beta}^n$ or the cone measure on $\Sph_{p,\beta}^{n-1}$. We denote by $\mathcal M(\R)$ the space of Borel probability measures on $\R$ equipped with the topology of weak convergence. On this topological space we consider the Borel $\sigma$-algebra, denoted by $\mathscr B(\mathcal M(\R))$.

\begin{thm}\label{theo:main}
Fix $0< p < +\infty$ and $\beta\in\{1,2,4\}$. For every $n\in\N$, let $Z_n$ be a random matrix chosen according to the uniform distribution on $\B_{p,\beta}^n$ or the cone measure on its boundary $\Sph_{p,\beta}^{n-1}$. Then the sequence of random probability measures
\[
\mu_n := \frac{1}{n}\sum_{i=1}^n \delta_{n^{1/p}\lambda_i(Z_n)},
\qquad n\in\N,
\]
satisfies an LDP on $\mathcal M(\R)$ with speed $n^2$ and good rate function $\mathscr I:\mathcal M(\R) \to [0,+\infty]$ defined by
\begin{equation}\label{eq:J_def_rate_funct}
\mathscr I(\mu) :=
\begin{cases}
- \frac \beta 2 \int_{\R}\int_{\R} \log|x-y| \, \mu(\dint x)\,\mu(\dint y) + \frac{\beta}{2p} \log\left(\frac{\sqrt{\pi}p \Gamma(\frac{p}{2})}{2^p\sqrt{e}\Gamma(\frac{p+1}{2})}\right) &\,: \int_{\R}|x|^p\mu(\dint x) \leq 1\\
+\infty &\,: \int_{\R}|x|^p\mu(\dint x) > 1\,.
\end{cases}
\end{equation}
\end{thm}

\begin{rmk}
(i) As we shall see later, the distribution of the eigenvalues of a point chosen uniformly at random in $\B_{p,\beta}^n$ can be related to the $1$-dimensional Coulomb gas (whose density is given in Equation \eqref{eq:distribution of (X_1,...,X_n)} below) of $n$ particles at inverse temperature $\beta>0$ in an external potential $V:t\mapsto |t|^p$ acting on each particle. \\
\noindent (ii) It can be seen from Equation \eqref{eq:B_const_formula} below that the additive constant to the logarithmic energy is closely linked to the limit of the free energy whose precise value follows from results of potential theory.
\end{rmk}

In the next theorem, we consider the case $p=+\infty$.
\begin{thm}\label{theo:main_p_infty}
Fix $\beta\in\{1,2,4\}$. For every $n\in\N$, let $Z_n$ be a random matrix chosen according to the uniform distribution on $\B_{\infty,\beta}^n$ or the cone measure on its boundary $\Sph_{\infty,\beta}^{n-1}$. Then the sequence of random probability measures
\[
\mu_n := \frac{1}{n}\sum_{i=1}^n \delta_{\lambda_i(Z_n)},
\qquad n\in\N,
\]
satisfies an LDP on $\mathcal M(\R)$ with speed $n^2$ and good rate function $\mathscr I:\mathcal M(\R) \to [0,+\infty]$ defined by
\[
\mathscr I(\mu) =
\begin{cases}
- \frac \beta 2 \int_{\R}\int_{\R} \log|x-y| \, \mu(\dint x)\,\mu(\dint y)
 - \frac {\beta}{2} \log 2 &\,: \mu ([-1,1]) =1 \\
+\infty &\,: \mu ([-1,1]) < 1\,.
\end{cases}
\]
\end{thm}

As a corollary, we can derive a law of large numbers for $\mu_n$ and show that (weakly almost surely) the sequence of empirical measures converges to a non-random limiting distribution given by the Ullman measure (for $p<+\infty$) or arcsine measure (for $p=+\infty$). We denote the weak convergence of probability measures by ${\overset{w}{\longrightarrow}}$.
\begin{cor}\label{cor:LLN}
Fix $0< p <+\infty$ and $\beta\in\{1,2,4\}$.  Let $\mu_{\infty}^{(p)}$ be the probability measure on the interval $[-b_p,b_p]$ with the Ullman density $x\mapsto h_p(x/b_p)/b_p$, where
\begin{equation}\label{eq:ullman_def}
h_p(x):={p\over\pi}\left(\int_{|x|}^1{t^{p-1}\over\sqrt{t^2-x^2}}\,\dint t\right) \ind_{\{|y|\leq 1\}}(x),
\qquad
b_p
:=
\left(\frac {p \sqrt \pi  \Gamma \left(\frac p2\right)} {\Gamma\left(\frac{p+1}{2}\right)}\right)^{1/p}.
\end{equation}
(The normalization is chosen so that the $p$-th moment of $\mu_{\infty}^{(p)}$ equals $1$).  Then the random measures $\mu_n$ defined in Theorem~\ref{theo:main} satisfy
$$
\Pro \left[\mu_n \toweak \mu_\infty^{(p)}\right] = 1.
$$
The result also holds in the case $p=+\infty$ with $\mu_\infty^{(\infty)}$ being the arcsine distribution with Lebesgue density $\frac {1}{\pi} (1-t^2)^{-1/2}$, $t\in (-1,1)$.
\end{cor}
For example, for $p=2$ and $p=1$ the limiting spectral density takes the following form, see~\cite[pp.~195--196]{hiai_petz},
$$
\frac{\mu_{\infty}^{(1)}(\dint x)}{\dint x} =  \frac 1 {\pi^2}\log \frac{\pi + \sqrt{\pi^2-x^2}}{|x|}\, \ind_{(-\pi,\pi)}(x),
\qquad
\frac{\mu_{\infty}^{(2)}(\dint x)}{\dint x} = \frac{1}{2\pi} \sqrt{4-x^2} \, \ind_{(-2,2)}(x),
$$
see Figure \ref{fig1}.

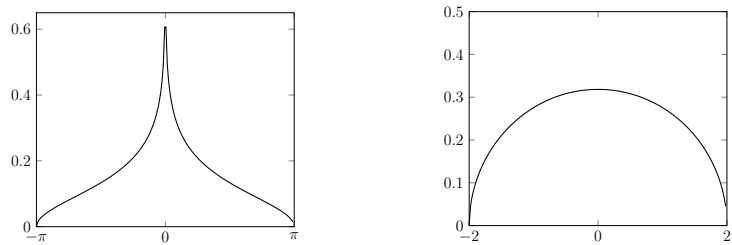
\begin{figure}[t]
\begin{center}
  \begin{tikzpicture}[scale=0.5]
      \begin{axis}[
       clip=false,
       xmin=-3.1415,xmax=3.1415,
       ymin=0, ymax=0.65,
       xtick={-3.1415,0,3.1415},
       xticklabels={$-\pi$, $0$,$\pi$}
       ]
        \addplot[domain=-3.1415:3.1415,samples=200,black,thick]{1/3.1415^2*ln((3.1415+sqrt(3.1415^2-x^2))/abs(x))};
      \end{axis}
    \end{tikzpicture}
    \qquad\qquad
  \begin{tikzpicture}[scale=0.5]
    \begin{axis}[
     clip=false,
     xmin=-2,xmax=2,
     ymin=0, ymax=0.5,
     xtick={-2,0,2},
     xticklabels={$-2$, $0$,$2$}
     ]
      \addplot[domain=-2:2,samples=200,black,thick]{1/(2*3.1415)*sqrt(4-x^2)};
    \end{axis}
  \end{tikzpicture}
\end{center}
\caption{Plots of the densities $\frac{\mu_{\infty}^{(1)}(\dint x)}{\dint x}$ (left panel) and $\frac{\mu_{\infty}^{(2)}(\dint x)}{\dint x}$ (right panel).}
\label{fig1}
\end{figure}

\subsection{Results for the non self-adjoint case}

After having discussed our main results for the self-adjoint case, we turn now to the non self-adjoint case, where the eigenvalues are replaced by the singular values. For an $(n\times n)$-matrix $A\in\Mat_n(\mathbb{F}_\beta)$ with entries from the skew field $\mathbb{F}_\beta$ with $\beta\in\{1,2,4\}$ we denote by $s_1(A),\ldots,s_n(A)$ the singular values of $A$. If $\beta=1$ or $\beta=2$ these are the eigenvalues of $\sqrt{AA^*}$, while in the Hamiltonian case $\beta=4$ we refer to \cite[Corollary E.13]{AGZ2010} for a formal definition. For $0<p\leq +\infty$ the Schatten $p$-ball is defined as
$$
\SSS\B_p^n := \Big\{A\in\Mat_n(\mathbb{F}_\beta):\sum_{j=1}^n|s_j(A)|^p\leq 1\Big\}
$$
with the convention that the sum is replaced by $\max\{|s_j|:j=1,\ldots,n\}$ in the case that $p=+\infty$. As in the self-adjoint case, $\Mat_n(\mathbb{F}_\beta)$ can be supplied with the structure of a Euclidean space in such a way that the 
$\beta n^2$-dimensional Hausdorff measure restricted to $\SSS\B_p^n$ is finite and can thus be normalized to a probability measure. Moreover, one can also define the cone probability measure on the boundary $\partial\SSS\B_p^n$ of $\SSS\B_p^n$.

The following Sanov-type LDP is the analogue of Theorem \ref{theo:main} and Theorem \ref{theo:main_p_infty} for the non self-adjoint case.

\begin{thm}\label{thm:MainNonSelfAdjoint}
Fix $\beta\in\{1,2,4\}$ and $0<p<+\infty$. For every $n\in\N$, let $Z_n$ be a random matrix chosen uniformly from $\SSS\B_p^n$ or according to the cone probability measure from $\partial\SSS\B_p^n$. Then the sequence of random probability measures
$$
\mu_n := {1\over n}\sum_{i=1}^n\delta_{n^{2/p}s_j^2(Z_n)},\qquad n\in\N,
$$
satisfies an LDP on the space $\mathcal{M}(\R_+)$ of Borel probability measures on $\mathbb{R}_+$, endowed with the weak topology, with speed $n^2$ and good rate function $\mathscr{J}:\mathcal{M}(\R_+)\to[0,+\infty]$ given by
$$
\mathscr{J}(\mu):=\begin{cases}
-{\beta\over 2}\int_{\R_+}\int_{\R_+}\log|x-y|\mu(\dint x)\mu(\dint y)+{\beta\over p}\log\Big({\sqrt{\pi}p\Gamma({p\over 2})\over 2^p\sqrt{e}\Gamma({p+1\over 2})}\Big) &: \int_{\R_+}|x|^{p/2}\mu(\dint x)\leq 1\\ 
+\infty &:\int_{\R_+}|x|^{p/2}\mu(\dint x)> 1.
\end{cases}
$$
The result continues to holds in the case $p=+\infty$ if the constant term in the rate function is replaced by its limiting value, as $p\to\infty$, which is given by $-{\beta\over 2}\log 2$.
\end{thm}

Again as a corollary, we derive a law of large numbers for the empirical singular-value distribution (and not for the squares of the singular values as in Theorem \ref{thm:MainNonSelfAdjoint}). This is the analogue of Corollary \ref{cor:LLN} for the Schatten $p$-balls.

\begin{cor}\label{cor:SLLNSchatten}
Fix $\beta\in\{1,2,4\}$ and $0<p<+\infty$. Let $\eta_\infty^{(p)}$ be the probability measure on $[0,b_p]$ with density $x\mapsto 2b_p^{-1}h_p(x/b_p)$ with $h_p(x)$ and $b_p$ given by \eqref{eq:ullman_def}. Further, for each $n\in\N$, let $Z_n$ be uniformly distributed in $\SSS\B_p^n$ or distributed according to the cone probability measure on $\partial\SSS\B_p^n$. Then 
$$
\Pro\Big[{1\over n}\sum_{j=1}^n\delta_{n^{1/p}s_j(Z_n)}\toweak\eta_\infty^{(p)}\Big] = 1.
$$
The result also holds in the case $p=+\infty$ with $\eta_\infty^{(\infty)}$ being the absolute arcsine distribution with density $x\mapsto{2\over\pi}(1-t^2)^{-1/2}$, $t\in(0,1)$.
\end{cor}

For example, if $p=1$ the limiting distribution has density 
$$
{\dint\eta_\infty^{(1)}\over \dint x}(x) = {2\over\pi^2}\log{\pi+\sqrt{\pi^2-x^2}\over x}\, \ind_{(0,\pi)}(x),
$$
and if $p=2$ we get the `quater-circle distribution' with density
$$
{\dint\eta_\infty^{(2)}\over \dint x}(x) = {1\over\pi}\sqrt{4-x^2} \,\ind_{(0,2)}(x),
$$
see Figure \ref{fig2}.

\begin{figure}[t]
\begin{center}
\begin{tikzpicture}[scale=0.5]
      \begin{axis}[
       clip=false,
       xmin=0,xmax=3.1415,
       ymin=0, ymax=1.3,
       xtick={0,3.1415},
       xticklabels={$0$,$\pi$}
       ]
        \addplot[domain=0:3.1415,samples=200,black,thick]{2/3.1415^2*ln((3.1415+sqrt(3.1415^2-x^2))/abs(x))};
      \end{axis}
    \end{tikzpicture}
    \qquad\qquad
  \begin{tikzpicture}[scale=0.5]
    \begin{axis}[
     clip=false,
     xmin=0,xmax=2,
     ymin=0, ymax=0.65,
     xtick={0,2},
     xticklabels={$0$,$2$}
     ]
      \addplot[domain=0:2,samples=200,black,thick]{1/(3.1415)*sqrt(4-x^2)};
    \end{axis}
  \end{tikzpicture}
\end{center}
\caption{Plots of the densities $\frac{\eta_{\infty}^{(1)}(\dint x)}{\dint x}$ (left panel) and $\frac{\eta_{\infty}^{(2)}(\dint x)}{\dint x}$ (right panel).}
\label{fig2}
\end{figure}
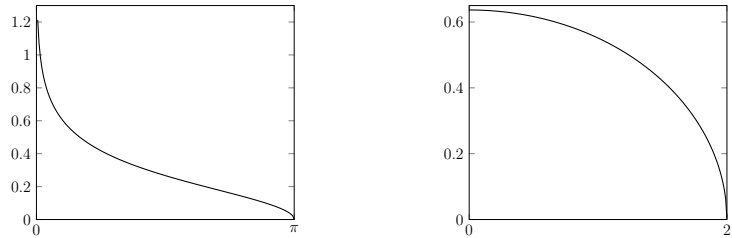

\subsection{Organization of the paper}
The rest of the paper is organized as follows. In Section \ref{sec:LD}, we introduce the concept of an LDP and some fundamental results used later. In Section \ref{sec:schechtmann_zinn}, we present a Schechtman-Zinn type probabilistic representation recently obtained in \cite{KPT2018a} and which is essential in our argumentation as well. The strategy of the proof is then outlined in Section \ref{Sec:strategy}. The remaining part of the manuscript is devoted to the proof of Theorem~\ref{theo:main} with exception of Section~\ref{sec:proof_LLN} in which we prove Corollary~\ref{cor:LLN}. The proof of Theorem~\ref{theo:main_p_infty} is omitted because it follows the same lines as the proof of Theorem~\ref{theo:main}. Since the proofs of Theorem \ref{thm:MainNonSelfAdjoint} and Corollary \ref{cor:SLLNSchatten} are very similar to the self-adjoint cases, we only sketch the differences in Section \ref{sec:ProofNonSelfAdjoint}.

\section{Large deviation principles} \label{sec:LD}
Keeping a brought readership from both probability theory and geometric analysis in mind, we provide in this section the necessary background material from the theory of large deviations, which may be found in \cite{DZ,dH,Kallenberg}, for example. We start directly with the definition of what is understood by a full and a weak large deviations principle. In this paper we fix an underlying probability space $(\Omega,\mathcal{F},\Pro)$, which we implicitly assume to be rich enough to carry all the random objects we consider. Also, for a subset $A$ of a topological (or metric) space we write $A^\circ$ and $\overline{A}$ for the interior and the closure of $A$, respectively.

\begin{df}
Let $(\xi_n)_{n\in\N}$ be a sequence of random elements taking values in some metric space $M$. Further, let $(s_n)_{n\in\N}$ be a positive sequence and $\mathscr{I}:M\to[0,+\infty]$ be a lower semi-continuous function.
We say that $(\xi_n)_{n\in\N}$ satisfies a (full) large deviations principle (LDP) with speed $s_n$ and a rate function $\mathscr{I}$ if
\begin{equation}\label{eq:LDPdefinition}
\begin{split}
-\inf_{x\in A^\circ}\mathscr{I}(x)
\leq\liminf_{n\to\infty}{1\over s_n}\log\Pro\left[\xi_n \in A \right]
\leq\limsup_{n\to\infty}{1\over s_n}\log\Pro\left[\xi_n \in A \right]\leq -\inf_{x\in\overline{A}}\mathscr{I}(x)
\end{split}
\end{equation}
for all Borel sets $A\subseteq M$. The rate function $\mathscr I$ is called good if its level sets  $\{x\in M\,:\, \mathscr{I}(x) \leq \alpha \}$ are compact for all $\alpha\geq 0$.
We say that $(\xi_n)_{n\in\N}$ satisfies a weak LDP with speed $s_n$ and rate function $\mathscr{I}$ if the upper bound in \eqref{eq:LDPdefinition} is valid only for compact sets $A\subseteq M$.
\end{df}

We notice that on the class of all $\mathscr{I}$-continuity sets, that is, on the class of Borel sets $A\subseteq M$ for which $\mathscr{I}(A^\circ)=\mathscr{I}(\bar{A})$ with $\mathscr{I}(A):=\inf\{\mathscr{I}(x):x\in A\}$, one has the exact limit relation
$$
\lim_{n\to\infty}{1\over s_n}\log\Pro\left[\xi_n \in A\right]=-\mathscr{I}(A)\,.
$$

What separates a weak from a full LDP is the so-called exponential tightness of the sequence of random variables (see, e.g., \cite[Lemma 1.2.18]{DZ} and \cite[Lemma 27.9]{Kallenberg}).

\begin{proposition}\label{prop:equivalence weak and full LDP}
Let $(\xi_n)_{n\in\N}$ be a sequence of random elements taking values in $M$. Suppose that it satisfies a weak LDP with speed $s_n$ and rate function $\mathscr{I}$. Then $(\xi_n)_{n\in\N}$ satisfies a full LDP if and only if the sequence is exponentially tight, that is, if and only if
$$
\inf_K\limsup_{n\to\infty}{1\over s_n}\log \Pro\left[\xi_n\notin K\right]=-\infty\,,
$$
where the infimum is running over all compact sets $K\subseteq M$. In this case, the rate function $\mathscr I$ is good.
\end{proposition}

The following proposition (see, for instance, \cite[Theorem 4.1.11]{DZ}) shows that it is sufficient to prove a weak LDP for a sequence of random elements for sets in a basis of the topology underlying $M$.

\begin{proposition}\label{prop:basis topology}
Let $\mathcal T$ be basis of the topology in a metric space $M$. Let $(\xi_n)_{n\in\N}$ be a sequence of $M$-valued random elements.  If for every $x\in M$,
$$
\mathscr I(x)
=
- \inf_{A\in\mathcal T:\, x\in A} \limsup_{n\to\infty} \frac 1 {s_n} \log \Pro \left[\xi_n \in A\right]
=
- \inf_{A\in\mathcal T:\, x\in A} \liminf_{n\to\infty} \frac 1 {s_n} \log \Pro \left[\xi_n \in A\right],
$$
then $(\xi_n)_{n\in\N}$ satisfies a weak LDP with speed $s_n$ and rate function $\mathscr{I}$.
\end{proposition}




It will be rather important for us to deduce from an already existing large deviations principle a new one by applying various transformations. We first consider the large deviations behavior under direct products. Let $M_1$ and $M_2$ be  metric spaces. Assume that $(\xi_n)_{n\in\N}$ is a sequence of $M_1$-valued random elements and that $(\eta_n)_{n\in\N}$ is a sequence of $M_2$-valued random elements. Assuming that both satisfy LDPs with the same speed, does then also the sequence $((\xi_n,\eta_n))_{n\in\N}$ of $M_1\times M_2$-valued random elements satisfy an LDP and, if so, what is its rate function? The following result can be found in \cite{APT2018}.

\begin{proposition}\label{JointRateFunction}
Assume that $(\xi_n)_{n\in\N}$ satisfies an LDP on a metric space $M_1$ with speed $s_n$ and good rate function $\mathscr{I}_{\xi}$ and that $(\eta_n)_{n\in\N}$ satisfies an LDP on a metric space $M_2$ with the same speed $s_n$ and good rate function $\mathscr{I}_\eta$. Then, if $\xi_n$ and $\eta_n$ are independent for every $n\in\N$, the sequence of pairs $((\xi_n,\eta_n))_{n\in\N}$ satisfies an LDP on $M_1\times M_2$ with speed $s_n$ and rate function
$$
\mathscr{I}_{(\xi,\eta)}(x)=\mathscr{I}_\xi(x_1)+\mathscr{I}_\eta(x_2),
\quad
x=(x_1,x_2)\in M_1\times M_2.
$$
\end{proposition}

Finally, we consider the possibility to `transport' a large deviations principle to another one by means of a continuous function. This device is known as the contraction principle and we refer to \cite[Theorem 4.2.1]{DZ} or \cite[Theorem 27.11(i)]{Kallenberg}.

\begin{proposition}\label{prop:contraction principle}
Let $M_1$ and $M_2$ be two metric spaces and let $F: M_1 \to M_2$ be a continuous function. Further, let $(\xi_n)_{n\in\N}$ be a sequence of $M_1$-valued random elements that satisfies an LDP with speed $s_n$ and good rate function $\mathscr{I}_\xi$. Then the sequence $(F(\xi_n))_{n\in\N}$ of $M_2$-valued random elements satisfies an LDP with the same speed and with good rate function $\mathscr{I} = \mathscr{I}_\xi\circ F^{-1}$, i.e.,
$$
\mathscr{I}(y):=\inf\{\mathscr{I}_\xi (x): x\in M_1, F(x)=y\},
\quad
y\in M_2,
$$
with the convention that $\mathscr{I}(y)=+\infty$ if $F^{-1}(\{y\})=\varnothing$.
\end{proposition}

\section{Distributional representation, free energy and strategy of the proof} 

We present here a Schechtman-Zinn type probabilistic representation for a random matrix chosen uniformly at random from $\B_{p,\beta}^n$ together with some results and concepts from potential theory used later. Also, before we continue with the technical details, we shall briefly explain the strategy of our proof.

\subsection{Distributional representation and free energy}\label{sec:schechtmann_zinn}

Let $Z_n$ be a random matrix uniformly distributed in the ball $\B_{p,\beta}^n$, $\beta\in\{1,2,4\}$. The basic fact we rely on is the following distributional representation of the empirical measure of $Z_n$, see~\cite[Corollary~4.3]{KPT2018a}:
\begin{equation}\label{eq:schechtman_zinn_uniform}
\mu_n
=
\frac 1n \sum_{i=1}^n \delta_{n^{1/p}\lambda_i(Z_n)}
\eqdistr
\frac 1n \sum_{i=1}^n \delta_{n^{1/p} U^{1 /\ell}{X_{i,n}\over\|X_n\|_p}},
\end{equation}
where $\eqdistr$ denotes equality in distribution and
\begin{itemize}
\item[(a)] $X_n = (X_{1,n},\dots,X_{n,n})$ is a random vector on $\R^n$ with joint Lebesgue density of the form
\begin{equation}\label{eq:distribution of (X_1,...,X_n)}
\frac 1 {C_{n,\beta,p}} e^{-n\sum\limits_{i=1}^n|x_i|^p} \prod_{1\leq i < j \leq n}|x_i-x_j|^{\beta},
\qquad
(x_1,\ldots,x_n)\in\R^n
\end{equation}
with a suitable normalization constant $C_{n,\beta,p}>0$ depending on $n$, $\beta$ and $p$,
\item[(b)] $U$ is a random variable with uniform distribution on $[0,1]$ that is independent of $X_n$,
\item[(c)] $\ell = \ell(n,\beta) =  {\beta n(n-1)\over 2} + \beta n$ is the (real) dimension of $\mathscr H_n(\mathbb{F}_\beta)$.
\end{itemize}
The above representation can be seen as a non-commutative counterpart to the representation of the uniform distribution on an the $\ell_p^n$-ball found by Schechtman and Zinn~\cite{SchechtmanZinn}. Similarly, if $Z_n$ is distributed according to the cone measure on $\Sph_{p,\beta}^{n-1}$, then
\begin{equation}\label{eq:schechtman_zinn_cone}
\mu_n
=
\frac 1n \sum_{i=1}^n \delta_{n^{1/p}\lambda_i(Z_n)}
\eqdistr
\frac 1n \sum_{i=1}^n \delta_{n^{1/p}{ X_{i,n}\over \|X_n\|_p}}.
\end{equation}

The distribution of $X_n$, known as ``matrix model''~\cite{PS2011}, has been intensely studied in the literature. Let us recall some results relevant to us.  The empirical distribution of $X_n$ is the random probability measure
$$
\nu_n := \frac 1n \sum_{i=1}^n \delta_{X_{i,n}}.
$$
We consider $\nu_n$ as a random element of the space $\mathcal M(\R)$ of Borel probability measures on $\R$ endowed with the weak topology.
For the following result we refer to ~\cite[Theorem 5.4.3]{hiai_petz}.
\begin{thm}\label{thm:LDP_nu_n}
The sequence $(\nu_n)_{n\in\N}$ satisfies a large deviations principle on $\mathcal M(\R)$ with speed $n^2$ and a good rate function
$$
\mathscr I_0(\mu)=
\begin{cases}
-  \frac \beta 2 \int_\R \int_\R \log |x-y| \mu(\dint x) \mu(\dint y) + \int_{\R} |x|^p \mu(\dint x) - B&\,: \int_{\R}|x|^p\mu(\dint x) <+\infty \\
+\infty &\,: \int_{\R}|x|^p\mu(\dint x) = +\infty,
\end{cases}
$$
where $B := \lim\limits_{n\to\infty} \frac1 {n^2}\log C_{n,\beta,p}$ with the constants $C_{n,\beta,p}$ given by \eqref{eq:distribution of (X_1,...,X_n)}.
\end{thm}

The constant $B$ in the theorem is the limit of the so-called free energy $ n^{-2}\log C_{n,\beta,p}$ as $n\to\infty$. Let us determine its precise value. By~\cite[Theorem 5.4.3]{hiai_petz}, the function $\mathscr I_0: \mathcal M(\R) \to [0,+\infty]$ has a unique minimizer $\mu_*^{(p)}$ satisfying $\mathscr I_0\big(\mu_*^{(p)}\big) = 0$, which means that
\[
B = -  \frac \beta 2 \int_\R \int_\R \log |x-y| \,\mu_*^{(p)}(\dint x)\, \mu_*^{(p)}(\dint y) + \int_{\R} |x|^p \mu_*^{(p)}(\dint x).
\]
By Proposition 5.3.4 from~\cite[p.~196]{hiai_petz}, the Lebesgue density of the minimizer is $x\mapsto h_p(x/r_p)/r_p$, where $h_p$ is the Ullman density given in~\eqref{eq:ullman_def}, and
$$
r_p = \left(\frac {\beta}{2p \alpha_p}\right)^{1/p} \qquad\text{and}\qquad
\alpha_p =
\frac {\Gamma\left(\frac{p+1}{2}\right)} {p \sqrt \pi  \Gamma \left(\frac p2\right)}
= \int_{-1}^1  |x|^p h_p(x) \dint x.
$$
Using this and the fact that by \cite[Proposition~5.3.4]{hiai_petz}
\[
\int_{\R}\int_\R \log |x-y| \mu_*^{(p)}(\dint x) \mu_*^{(p)}(\dint y) = \log \frac{r_p}2 - \frac 1 {2p}\,,
\]
we obtain the following explicit formula for the limiting free energy:
\begin{equation}\label{eq:B_const_formula}
B = \lim_{n\to\infty} \frac1 {n^2}\log C_{n,\beta,p} = \frac {\beta}{2p} \log \left(\frac{\beta \sqrt \pi \Gamma\left(\frac p2\right)}{2 \Gamma\left(\frac{p+1}{2}\right)}\right) - \frac \beta 2 \log 2 - \frac {3\beta}{4p}\,.
\end{equation}

\subsection{Strategy of the proof}\label{Sec:strategy}
We adopt the notation introduced in the previous sections. The outline of our proof of Theorem~\ref{theo:main} is as follows.

\vskip 2mm
\noindent
\textbf{Step 1:}
We prove an LDP for the sequence of pairs
\begin{equation}\label{eq:pair}
\left(\nu_n, \int_\R |x|^p \nu_n(\dint x)\right) = \left(\frac{1}{n}\sum_{i=1}^n\delta_{X_{i,n}},\frac{1}{n}\sum_{i=1}^n |X_{i,n}|^p\right),
\quad n\in\N,
\end{equation}
of empirical measures of the vector $X_n$ and empirical $p$th moments of these measures.
In a first attempt, it is natural to try to apply the contraction principle to the LDP for $\nu_n$ (stated above in Theorem \ref{thm:LDP_nu_n}) with the mapping $\mu \mapsto (\mu, \int_\R |x|^p \mu(\dint x))$. However, this mapping is not continuous in the weak topology.  At first sight, this may look like a merely technical issue, but it is not.  As we shall see, the correct rate function for the pair $(\nu_n, \int_\R |x|^p \nu_n(\dint x))$ does not coincide with what one would expect by a naive application of the contraction principle.


\vskip 2mm
\noindent
\textbf{Step 2:} Using the contraction principle, we derive an LDP for the sequence of random measures
\[
\frac{1}{n}\sum_{i=1}^n \delta_{n^{1/p}{ X_{i,n}\over \|X\|_p}},
\qquad n\in\N,
\]
thus proving Theorem~\ref{theo:main} in the case when $Z_n$ is sampled according to the cone measure on $\Sph_{p,\beta}^{n-1}$.

\vskip 2mm
\noindent
\textbf{Step 3:}
If $U$ is uniformly distributed on $[0,1]$, then the sequence $(U^{1/\ell})_{n\in\N}$, where we recall that $\ell = {\beta n(n-1)\over 2} + \beta n$, satisfies an LDP with rate function $x\mapsto -\frac{\beta}{2}\log x$ if  $x\in(0,1]$ and $+\infty$ otherwise. Applying the contraction principle, we derive an LDP for the sequence of random measures
\[
\frac 1n \sum_{i=1}^n \delta_{n^{1/p} U^{1 /\ell}{X_{i,n}\over\|X_n\|_p}}, \qquad n\in\N,
\]
thus proving Theorem~\ref{theo:main} in the case when $Z_n$ is distributed uniformly on  the ball $\B_{p,\beta}^{n}$.



\section{Step 1 -- A large deviations principle for the empirical pair}
The following theorem captures the large deviations behavior for the sequence of pairs of empirical measures $\nu_n$ of the vector $X_n$ and empirical $p$th moments of these measures. 

\begin{thm}\label{thm:ldp empirical pair}
Let $0 < p <+\infty$ and $\beta\in(0,+\infty)$. Assume that for each $n\in\N$, the random vector $X_n=(X_{1,n},\dots,X_{n,n})$ has joint density given by \eqref{eq:distribution of (X_1,...,X_n)}. Then the sequence of random elements
$$
\left(\nu_n,\int_\R |x|^p \nu_n(\dint x)\right) = \left(\frac{1}{n}\sum_{i=1}^n\delta_{X_{i,n}},\frac{1}{n}\sum_{i=1}^n |X_{i,n}|^p\right),
\quad n\in\N,
$$
satisfies an LDP with speed $n^2$ and good rate function $\mathscr I_1:\mathcal M(\R) \times [0,+\infty)\to [0,+\infty]$ given by
\begin{equation}\label{eq:def_I_1}
\mathscr I_1(\mu,m) :=
\begin{cases}
- \frac \beta 2 \int_{\R}\int_{\R} \log|x-y|\,\mu(\dint x)\,\mu(\dint y) + m + B &: m\geq \int_{\R}|x|^p\,\mu(\dint x) \\
+\infty & : m<\int_{\R}|x|^p\,\mu(\dint x),
\end{cases}
\end{equation}
where $B$ is the same constant as in~\eqref{eq:B_const_formula}.
\end{thm}

The proof is split into several parts which are considered in the subsections below. The main task is to establish a weak large deviations principle by verifying the conditions of Proposition~\ref{prop:basis topology} on the space $\mathcal M(\R)\times [0,+\infty)$ equipped with its product topology arising from the weak topology on $\mathcal M(\R)$ and the standard topology on $[0,+\infty)$. More precisely, our aim is to show that for every $(\mu,m)\in \mathcal M(\R) \times [0,+\infty)$, we have
\begin{align*}
&\inf_{G:\, (\mu,m)\in G} \limsup_{n\to\infty} \frac 1 {n^2} \log \Pro\left[\left(\nu_n, \int_{\R}|x|^p\,\nu_n(\dint x)\right)\in G\right] \leq - \mathscr I_1(\mu,m)
\end{align*}
and
\begin{align*}
&\inf_{G:\, (\mu,m)\in G} \liminf_{n\to\infty} \frac 1 {n^2} \log \Pro\left[\left(\nu_n, \int_{\R}|x|^p\,\nu_n(\dint x)\right)\in G\right] \geq -\mathscr  I_1(\mu,m),
\end{align*}
where the infimum is taken over $G$'s belonging to some neighborhood base of the pair $(\mu,m)$ with respect to the product topology just described. This will be done in Lemmas~\ref{lem:weak_LDP_m_smaller_moment}, \ref{lem:lower_bound}, and \ref{lem:upper_bound}.
The exponential tightness, needed to complete the proof of Theorem \ref{thm:ldp empirical pair}, will be verified in Lemma~\ref{lem:exp_tight}.

\subsection{The case \texorpdfstring{$m<\int_\R |x|^p \mu(\dint x)$}{m < int |x|p mu(dx)}}

We shall prove the following lemma, which then implies that the rate function $\mathscr I_1(\mu,m)$ of the empirical pair is $+\infty$ when $m$ is smaller than the $p$th moment of the measure $\mu$.

\begin{lemma}\label{lem:weak_LDP_m_smaller_moment}
Let $\mu\in \mathcal M(\R)$ and $m\in [0,+\infty)$ be such that
$
m < \int_{\R}|x|^p\,\mu(\dint x).
$
Then there is a neighborhood $G$ of $(\mu, m)$ in the product space $\mathcal M(\R) \times [0,+\infty)$ such that, for all $n\in\N$,
\[
\Pro\left[\left(\nu_n, \int_{\R}|x|^p\,\nu_n(\dint x) \right)\in G\right]=0.
\]
\end{lemma}
\begin{proof}
Our goal is to construct explicitly a neighborhood $\mathcal O(\mu)\subseteq \mathcal M(\R)$ of $\mu$ (in the weak topology) and a neighborhood $\mathcal O(m)\subseteq [0,+\infty)$ of $m$  such that
\[
\left\{\nu \in \mathcal M(\R):  \left(\nu,\int_{\R}|x|^p\,\nu(\dint x) \right)\in \mathcal O(\mu) \times \mathcal O(m) \right\}=\varnothing.
\]
Putting $G:=\mathcal{O}(\mu)\times\mathcal{O}(m)$ then completes the proof of the lemma.

\vspace*{2mm}
\noindent
\textit{Case 1.}
Assume first that $\int_{\R}|x|^p\,\mu(\dint x) < +\infty$. Since $m < \int_{\R}|x|^p\,\mu(\dint x)$,
there exists some $\alpha>0$ such that
\begin{equation}\label{eq:m_plus_alpha}
m+\alpha = \int_{\R}|x|^p\,\mu(\dint x).
\end{equation}
By monotone convergence, we can find $A>0$ (sufficiently large) such that
\begin{equation}\label{eq:distance of cut-off A to mu moment}
\left|\int_{\R} \min\{A,|x|^p\} \,\mu(\dint x) - \int_{\R} |x|^p \,\mu(\dint x)\right| < \frac \alpha 3.
\end{equation}
Now we define the weak neighborhood of $\mu$ as
\[
\mathcal O(\mu) := \left\{ \nu \in\mathcal M(\R)\,:\, \left| \int_{\R} \min\{A,|x|^p\} \,\nu(\dint x) - \int_{\R} \min\{A,|x|^p\} \,\mu(\dint x)\right| < \frac \alpha 3 \right\}.
\]
Indeed, this is a neighborhood in the weak topology because the function $x\mapsto \min\{A,|x|^p\}$ is continuous and bounded. Assume now that $\nu \in \mathcal O(\mu)$. Then it follows from~\eqref{eq:distance of cut-off A to mu moment} and the triangle inequality that
\[
\left|\int_{\R} \min\{ A,|x|^p\} \,\nu(\dint x) - \int_{\R}|x|^p\,\mu(\dint x) \right| \leq \frac {2\alpha}{3}.
\]
From this and~\eqref{eq:m_plus_alpha} we infer
\[
\int_{\R} |x|^p \,\nu(\dint x) \geq \int_{\R} \min\{ A,|x|^p\} \,\nu(\dint x) \geq
\int_{\R}|x|^p\,\mu(\dint x) -\frac {2\alpha}{3}
=
m + \frac \alpha 3,
\]
which implies that $\int_{\R}|x|^p\,\nu(\dint x) \notin \mathcal O(m)$ if we define $\mathcal O (m) := (m - \frac \alpha6, m + \frac \alpha 6) \cap [0,+\infty)$.

\vspace*{2mm}
\noindent
\textit{Case 2.}
The case $\int_{\R}|x|^p\,\mu(\dint x) = +\infty$ is similar. By monotone convergence, we can find a sufficiently large $A>0$ such that
\begin{equation}\label{eq:distance of cut-off A to mu moment1}
\int_{\R} \min\{A,|x|^p\} \,\mu(\dint x) > 3m.
\end{equation}
Consider the following weak neighborhood of $\mu$ defined as
\[
\mathcal O(\mu) := \left\{ \nu \in\mathcal M(\R)\,:\, \left| \int_{\R} \min\{A,|x|^p\} \,\nu(\dint x) - \int_{\R} \min\{A,|x|^p\} \,\mu(\dint x)\right| < m \right\}.
\]
It follows directly from \eqref{eq:distance of cut-off A to mu moment1} that every $\nu \in \mathcal O(\mu)$  satisfies
$$
\int_{\R} \min\{A,|x|^p\} \,\nu(\dint x) > 2m.
$$
In particular, we have $\int_{\R}|x|^p\,\nu(\dint x) > 2m$. Defining the neighborhood $\mathcal O (m) := (\frac 12 m , \frac 32 m)$, it follows that $\int_{\R}|x|^p\nu(\dint x) \notin \mathcal O(m)$.
\end{proof}

\subsection{The case \texorpdfstring{$m\geq \int_\R |x|^p \mu(\dint x)$}{m >= int |x|p mu(dx)}}

We consider here the case where $m\geq\int_\R |x|^p \mu(\dint x)$. We prove lower and upper bounds in the weak LDP separately. Let us start with the lower bound.

\subsubsection{The lower bound} 
First we note that since we are in the case $m\geq \int_{\R}|x|^p\,\mu(\dint x)$, if we assume $m=0$, then necessarily $\mu=\delta_0$ (the Dirac measure at $0$) and so $\mathscr I_1(\delta_0,0) = +\infty$. In particular, for any neighborhood $G$ of $(\delta_0,0)$,
\begin{align*}
\liminf_{n\to\infty} \frac{1}{n^2} \log \Pro\left[\left(\nu_n, \int_{\R}|x|^p \nu_n(\dint x)\right)\in G\right]
\geq
-\infty = -\mathscr I_1(\delta_0,0),
\end{align*}
trivially holds. We therefore assume from now on that $m>0$. Our aim is to prove the following result.

\begin{lemma}\label{lem:lower_bound}
Let $\mu\in\mathcal M(\R)$ and $m\in (0,+\infty)$ be such that $m > \int_\R |x|^p \mu(\dint x)$. Then, for every neighborhood $G$ of $(\mu,m)$ in the space $\mathcal M(\R) \times [0,+\infty)$, we have
\begin{align*}
\liminf_{n\to\infty} \frac{1}{n^2} \log&\, \Pro\left[\left(\nu_n, \int_{\R}|x|^p \nu_n(\dint x)\right)\in G\right]\\
&\geq
\frac \beta 2 \int_{\R}\int_{\R} \log|x-y|\,\mu(\dint x)\,\mu(\dint y) - B - m.
\end{align*}
\end{lemma}

We may assume that the neighborhood $G$ is of the form
\[
G = \mathcal O_{\varepsilon,d}(\mu)\times \mathcal (m-\delta,m+\delta)
\]
with $\varepsilon>0$, $\delta>0$, $d\in\N$ and
\begin{align*}
\mathcal O_{\varepsilon,d}(\mu) &= \mathcal O_{\varepsilon,f_1,\dots,f_d}(\mu)\\
&=\left\{\nu \in\mathcal M(\R)\,:\, \left| \int_\R f_i(y) \,\mu(\dint y) - \int_\R f_i(y) \,\nu(\dint y)  \right| < \varepsilon\;\;\forall i=1,\dots,d \right\}
\end{align*}
for some  functions $f_1,\dots,f_d\in\mathcal C_b(\R)$. Here and in what follows, $\mathcal C_b(\R)$ denotes the space of bounded continuous functions on $\R$.  We may further assume that $\|f_i\|_\infty\leq 1$ for all $i\in\{1,\dots,d\}$.

Before proceeding to precise statements, let us give a heuristic explanation of what follows. We need a lower bound on the probability of the event
$$
\left\{\nu_n\in \mathcal O_{\varepsilon,d}(\mu), \int_\R |x|^p \nu_n(\dint x) \in (m-\delta,m+\delta)\right\}.
$$
At first sight, the conditions $\nu_n\in \mathcal O_{\varepsilon,d}(\mu)$ and $\int_\R |x|^p \nu_n(\dint x) \in (m-\delta,m+\delta)$ seem to contradict each other for sufficiently small $\eps>0$ and $\delta>0$. Indeed, the first condition states that $\nu_n\approx \mu$, which seems to imply that $\int_\R |x|^p \nu_n(\dint x) \approx \int_\R |x|^p \mu(\dint x)$, which in turn  contradicts the condition $\int_\R |x|^p \nu_n(\dint x) \approx m$.   However, as we already mentioned above, the map $\nu \mapsto \int_\R |x|^p \nu(\dint x)$ is not weakly continuous, so that this argumentation is incorrect.

As we are aiming for a lower bound, we shall provide an explicit description of how the event $\{\nu_n\in \mathcal O_{\varepsilon,d}(\mu), \int_\R |x|^p \nu_n(\dint x) \in (m-\delta,m+\delta)\}$ can be realized.
The next lemma states that this event occurs if the following three events $A_n$, $B_n$ and $C_n$ occur simultaneously:
\begin{itemize}
\item [$A_n$:] The one-leave-out empirical measure
\[
\nu_{n-1}':=\frac{1}{n-1}\sum_{i=1}^{n-1}\delta_{X_{i,n}}
\]
is ``close'' to $\mu$ in the weak topology,
\item [$B_n$:] the $p$th moment of $\nu_{n-1}'$ is ``close'' to the $p$-th moment of $\mu$,
\item [$C_n$:] the last element $X_{n,n}$ is an outlier, which is ``close'' to $$n^{1/p}\Big(m-\int_\R |x|^p \mu(\dint x)\Big).$$
\end{itemize}
The event $C_n$ is crucial because it ensures that $\int_\R |x|^p \nu_n(\dint x)$ is ``close'' to $m$ rather than to the strictly smaller number $\int_\R |x|^p \mu(\dint x)$.
The precise statement is as follows.

\begin{lemma}\label{lem:smaller_event}
Let $\varepsilon,\delta>0$, let $\mu\in\mathcal M(\R)$ and assume that $m\in (0,+\infty)$ is such that
$
m > \int_\R |x|^p \mu(\dint x)
$.
Let us define the interval
\begin{equation}\label{eq:def_D_n}
D_n:= \left(n^{1/p}\left(m-\int_{\R}|y|^p\,\mu(\dint y)\right)^{1/p}-n^{1/p-2},
\;n^{1/p}\left(m-\int_{\R}|y|^p\,\mu(\dint y)\right)^{1/p}\right).
\end{equation}
Then, for all sufficiently large $n\in\N$,
\begin{align*}
& \Pro\left[ \nu_n \in \mathcal O_{\varepsilon,d}(\mu), \int_\R |x|^p \nu_n(\dint x) \in  (m-\delta,m+\delta) \right] \cr
& \geq \Pro\left[\nu'_{n-1}\in \mathcal O_{\eps/3,d}(\mu),\; \left|\int_{\R}|y|^p\,\nu'_{n-1}(\dint y)-\int_{\R}|y|^p\,\mu(\dint y)\right| < \frac{\delta}{3},\; X_{n,n}\in D_n \right].
\end{align*}
\end{lemma}
\begin{proof}
Recall the definition of $\mathcal{O}_{\eps,d}(\mu)$ (in particular the functions $f_1,\ldots,f_d$ involved there) and consider the events
\begin{align}
A_n &:= \left\{\nu'_{n-1}\in \mathcal O_{\eps/3,d}(\mu)\right\} \label{eq:def_A_n},
\\
B_n &:=\left\{\left|\int_{\R}|y|^p\,\nu'_{n-1}(\dint y)-\int_{\R}|y|^p\,\mu(\dint y)\right| < \frac{\delta}{3}\right\},\label{eq:def_B_n}
\\
C_n &:=\left\{X_{n,n}\in D_n \right\}. \label{eq:def_C_n}
\end{align}

\vspace*{2mm}
\noindent
\textit{Step 1.}
We prove that on the event $A_n\cap B_n\cap C_n$ it holds that $\nu_{n}\in \mathcal O_{\varepsilon,d}(\mu)$.
Indeed, for all $i\in\{1,\dots,d\}$, we observe that
\begin{align*}
\left|\int_{\R} f_i(y) \,\nu_n(\dint y) - \int_{\R} f_i(y) \,\mu(\dint y)\right|
& = \left|\frac{1}{n}\sum_{j=1}^n f_i(X_{j,n}) - \int_{\R} f_i(y) \,\mu(\dint y)\right| \cr
& = \left|\frac{1}{n}f_i(X_{n,n}) + \frac{1}{n}\sum_{j=1}^{n-1} f_i(X_{j,n}) - \int_{\R} f_i(y) \,\mu(\dint y)\right| \cr
& = \left|\frac{1}{n}f_i(X_{n,n}) + \frac{n-1}{n}\int_{\R} f_i(y) \,\nu'_{n-1}(\dint y) - \int_{\R} f_i(y) \,\mu(\dint y)\right| \cr
& \leq  \frac{1}{n} + \left| \frac{n-1}{n}\int_{\R} f_i(y) \,\nu'_{n-1}(\dint y) - \int_{\R} f_i(y) \,\mu(\dint y) \right|,
\end{align*}
where in the last estimate we used the triangle inequality and that $\|f_i\|_\infty \leq 1$ for all $i\in\{1,\dots,d\}$.
Using again the triangle inequality (now in the second step), the assumption $\nu'_{n-1}\in \mathcal O_{\eps/3,d}(\mu)$ and the fact that $\|f_i\|_\infty \leq 1$ for all $i\in\{1,\dots,d\}$, we arrive at
\begin{align*}
& \left| \frac{n-1}{n}\int_{\R} f_i(y) \,\nu'_{n-1}(\dint y) - \int_{\R} f_i(y) \,\mu(\dint y) \right| \cr
& {\leq} \left| \frac{n-1}{n}\int_{\R} f_i(y) \,\nu'_{n-1}(\dint y) - \int_{\R} f_i(y) \,\nu'_{n-1}(\dint y) \right|+ \left|\int_{\R} f_i(y) \,\nu'_{n-1}(\dint y) - \int_{\R} f_i(y) \,\mu(\dint y) \right| \cr
& \leq  \frac 1n \left|\int_{\R} f_i(y) \,\nu'_{n-1}(\dint y) \right| + \frac{\varepsilon}{3}
 \leq  \frac 1n + \frac{\varepsilon}{3}.
\end{align*}
Obviously, we may choose $n_0\in\N$ such that for all $n\geq n_0$ we have $\frac{1}{n} \leq \frac \varepsilon 3$. Putting everything together, we obtain for all $n\geq n_0$ and $i\in \{1,\ldots,d\}$ that
\[
\left|\int_{\R} f_i(y) \,\nu_n(\dint y) - \int_{\R} f_i(y) \,\mu(\dint y)\right| < {1\over n}+{1\over n}+{\eps\over 3}=\varepsilon.
\]
In other words, this means that $\nu_n\in \mathcal O_{\varepsilon,d}(\mu)$.

\vspace*{2mm}
\noindent
\textit{Step 2.}
Next we prove that on the event $A_n\cap B_n\cap C_n$, we have $\int_\R |x|^p \nu_n(\dint x) \in(m-\delta,m+\delta)$.
Note that we may write
\[
\int_\R |x|^p \nu_n(\dint x) = \frac{n-1}{n} \int_{\R}|y|^p\,\nu'_{n-1}(\dint y) + \frac{1}{n}|X_{n,n}|^p.
\]
In order to assure that $\int_\R |x|^p \nu_n(\dint x) \in(m-\delta,m+\delta)$
it is therefore enough to have
\begin{align}\label{eq:assumption integral vs |X_n|^p-m}
\left|\int_{\R}|y|^p\,\mu(\dint y)+ \frac{1}{n}|X_{n,n}|^p-m \right| &\leq \frac{\delta}{3}.
\end{align}
Indeed, using the triangle inequality in the second, the definition \eqref{eq:def_B_n} of $B_n$ and \eqref{eq:assumption integral vs |X_n|^p-m} in the third and once again the definition \eqref{eq:def_B_n} of $B_n$ in the fourth step, would imply that
\begin{eqnarray*}
\lefteqn{\left|\int_\R |x|^p \nu_n(\dint x) - m\right|}\\
&=&
\left|\frac{n-1}{n} \int_{\R}|y|^p\,\nu'_{n-1}(\dint y) + \frac{1}{n}|X_{n,n}|^p-m\right| \cr
&{\leq}&
\left|\frac{n-1}{n} \int_{\R}|y|^p\,\nu'_{n-1}(\dint y) -\int_{\R}|y|^p\,\nu'_{n-1}(\dint y)\right| +\left|\int_{\R}|y|^p\,\nu'_{n-1}(\dint y) - \int_{\R}|y|^p \,\mu(\dint y) \right|
\cr
&~& \qquad + \left|\int_{\R}|y|^p \,\mu(\dint y) + \frac{1}{n}|X_{n,n}|^p - m\right| \cr
&{\leq}&
\frac 1n \int_\R|y|^p\,\nu'_{n-1}(\dint y) + \frac{\delta}{3}+\frac{\delta}{3}  \cr
& {\leq}&  \frac 1n \left(\int_\R|y|^p\,\mu(\dint y)+\frac{\delta}{3}\right) + \frac{\delta}{3}+\frac{\delta}{3} \cr
& \leq&  \delta,
\end{eqnarray*}
as we may choose $n_1\in\N$ such that for all $n\geq n_1$,
\[
\frac 1n \left(\int_\R|y|^p\,\mu(\dint y)+\frac{\delta}{3}\right) \leq \frac{\delta}{3}.
\]
It remains to prove that~\eqref{eq:assumption integral vs |X_n|^p-m} holds on the event $A_n\cap B_n\cap C_n$.
By definition \eqref{eq:def_C_n} of the event $C_n$ and the definition \eqref{eq:def_D_n} of the interval $D_n$, we have
$$
\frac{1}{n}|X_{n,n}|^p - \left(m-\int_{\R}|y|^p\,\mu(\dint y)\right) \leq 0.
$$
On the other hand, the same definitions of the event $C_n$ and the interval $D_n$ imply that
\begin{align*}
\lefteqn{\frac{1}{n}|X_{n,n}|^p - \left(m-\int_{\R}|y|^p\,\mu(\dint y)\right)} \cr
&\geq
\left(m-\int_{\R}|y|^p\,\mu(\dint y)\right)\cdot\left[ \left(1 - \frac{1}{n^2}\left(m-\int_{\R}|y|^p\,\mu(\dint y)\right)^{-1/p}\right)^p-1\right] \cr
&\geq
-\frac \delta 3
\end{align*}
for all sufficiently large $n\in\N$.
This concludes the proof of~\eqref{eq:assumption integral vs |X_n|^p-m}.
\end{proof}

We shall now prove the lower bound for the probability of the previously discussed event $A_n\cap B_n\cap C_n$, which is based on a decomposition technique. For now, we restrict ourselves to probability measures which are supported on a compact interval and have continuous density on this interval. As we shall see later, it is actually enough to consider this case.

\begin{lemma}\label{lem:lower bound pair compactly supported}
Let $\mu \in\mathcal M(\R)$ be supported on an interval $[a,b]$ with $-\infty<a<b<+\infty$ and let $m\in (0,+\infty)$ be such that
$m > \int_\R |x|^p \mu(\dint x)$. Assume that $\mu$ has Lebesgue density $h$ that is continuous on $[a,b]$ and satisfies $\inf_{x\in[a,b]} h(x) > 0$. Let $d\in\N$, $f_1,\ldots,f_d\in \mathcal C_b(\R)$ and $\varepsilon,\delta>0$.  Then
\begin{align*}
\liminf_{n\to\infty} \frac{1}{n^2}\log &\,\Pro \left[\nu'_{n-1}\in\mathcal O_{\varepsilon,d}(\mu), \;\left|\int_{\R}|y|^p\,\nu'_{n-1}(\dint y)-\int_{\R}|y|^p\,\mu(\dint y)\right| < \delta,
\; X_{n,n}\in D_n \right] \\
& \geq \frac \beta 2 \int_{\R}\int_{\R} \log|x-y| \,\mu(\dint x)\, \mu(\dint y) -B-m.
\end{align*}
\end{lemma}
\begin{proof}
We use a classical approach similar to the proof of Lemma 5.4.6 in~\cite[pp.~216--217]{hiai_petz}, but the ``outlier'' $X_{n,n}$ needs to be  treated differently.
Denote by $g:[0,1]\to[a,b]$ the inverse function of the continuous, strictly monotone function
\[
t\mapsto \int_a^th(x)\,\dint x, \quad t\in [a,b].
\]
Note that $g(0)=a$ and $g(1)=b$.  For each $k\in\{1,\dots,n-1\}$, we define
\[
a_k^{(n-1)} := g\left(\frac{2k-1}{2(n-1)}\right)\quad\text{and}\quad b_k^{(n-1)}:=g\left(\frac{2k}{2(n-1)}\right).
\]
Note that
\[
a = b_0^{(n-1)} < a_1^{(n-1)} < b_1^{(n-1)} < a_2^{(n-1)} < b_2^{(n-1)} < \ldots < a_{n-1}^{(n-1)} < b_{n-1}^{(n-1)}=b.
\]
This way we obtain a decomposition of the interval $[a,b]$ into $2(n-1)$ intervals. Let
\[
\Delta_{n-1} := \left\{(t_1,\dots,t_{n-1})\in\R^{n-1}\,:\, a_{k}^{(n-1)} \leq t_k \leq b_{k}^{(n-1)}\quad \text{for all $k=1,\dots,n-1$} \right\}.
\]
For sufficiently large $n\in\N$, we have that, for all $t=(t_1,\ldots,t_n)\in\Delta_{n-1}$,
\[
\frac{1}{n-1}\sum_{k=1}^{n-1}\delta_{t_k} \in \mathcal O_{\varepsilon,d}(\mu)\quad\text{and}\quad \left|\frac{1}{n-1}\sum_{k=1}^{n-1}|t_k|^p - \int_{\R}|x|^p\,\mu(\dint x) \right|<\delta.
\]
Therefore, recalling from~\eqref{eq:distribution of (X_1,...,X_n)} the density of the random vector $X_n$, we can write
\begin{align*}
& \Pro \left[\nu'_{n-1}\in\mathcal O_{\varepsilon,d}(\mu),\; \left|\int_{\R}|y|^p\,\nu'_{n-1}(\dint y)-\int_{\R}|y|^p\,\mu(\dint y)\right| < \delta, \; X_{n,n}\in D_n \right] \cr
& \geq \frac{1}{C_{n,\beta,p}} \int_{\Delta_{n-1}\times D_n} e^{-n\sum_{k=1}^n|t_k|^p}\prod_{1\leq i<j\leq n} |t_j-t_i|^{\beta} \,\dint(t_1,\ldots,t_n) \cr
& = \frac{1}{C_{n,\beta,p}} \int_{\Delta_{n-1}} \int_{D_n} e^{-n|t_n|^p}e^{-n\sum_{k=1}^{n-1}|t_k|^{p}} \prod_{1\leq i<j<n} |t_j-t_i|^{\beta} \prod_{1\leq i< n} |t_n-t_i|^{\beta} \,\dint t_n\dint(t_1,\ldots,t_{n-1}).
\end{align*}
Now we define, for each $k\in\{1,\dots,n-1\}$,
\[
\xi_k^{(n-1)}:= \max\left\{|t|^p \,:\, t\in\left[a_k^{(n-1)},b_k^{(n-1)}\right] \right\}.
\]
Then
\begin{align*}
& \Pro \left[\nu'_{n-1}\in\mathcal O_{\varepsilon,d}(\mu),\; \left|\int_{\R}|y|^p\,\nu'_{n-1}(\dint y)-\int_{\R}|y|^p\,\mu(\dint y)\right| < \delta, \; X_{n,n}\in D_n \right] \cr
& \geq  \frac{1}{C_{n,\beta,p}} \int_{\Delta_{n-1}} \int_{D_n} e^{-n|t_n|^p}e^{-n\sum_{k=1}^{n-1}|\xi^{(n-1)}_k|^{p}} \prod_{1\leq i<j<n} |a_j^{(n-1)}-b_i^{(n-1)}|^{\beta}\\
&\hspace{5cm}\times \prod_{1\leq i< n} |t_n-t_i|^{\beta} \,\dint t_n\dint(t_1,\ldots,t_{n-1})\cr
& \geq \frac{1}{C_{n,\beta,p}}e^{-n\sum_{k=1}^{n-1}|\xi^{(n-1)}_k|^{p}} \prod_{1\leq i<j<n} |a_j^{(n-1)}-b_i^{(n-1)}|^{\beta} e^{-n^2(m-\int_{\R}|x|^p\,\mu(\dint x))}\vol_n(\Delta_{n-1}\times D_n),
\end{align*}
where we used that on $\Delta_{n-1}\times D_n$ we have $|t_n|^p < n (m-\int_{\R}|x|^p\,\mu(\dint x))$ by~\eqref{eq:def_D_n}, and
\begin{equation}\label{eq:prod_t_n_diff}
\prod_{1\leq i< n} |t_n-t_i|^{\beta} \geq 1
\end{equation}
for sufficiently large $n\in\N$. Indeed, on $\Delta_{n-1}\times D_n$ we have $t_1,\ldots,t_{n-1} \in [a,b]$, whereas $t_n > n^{1/p} (m-\int_{\R}|y|^p\,\mu(\dint y))^{1/p}-n^{1/p-2}$ by~\eqref{eq:def_D_n}, which exceeds $b+1$ for sufficiently large $n$. This proves~\eqref{eq:prod_t_n_diff}.
Observe also that
\[
\vol_n(\Delta_{n-1}\times D_n) \geq n^{1/p-2}  \prod_{k=1}^{n-1}(b_{k}^{(n-1)}- a_{k}^{(n-1)})  \geq \left(\frac1 {(2n-2){\sup_{x\in [a,b]} h(x)}}\right)^{n-1}n^{1/p-2},
\]
which implies 
$$
\liminf_{n\to\infty} \frac 1 {n^2} \log \vol_n(\Delta_{n-1}\times D_n) \geq 0.
$$
 Recall also that $\lim_{n\to\infty} \frac1 {n^2}\log C_{n,\beta,p}=B$.
Using that
\begin{align*}
&\lim_{n\to\infty} \frac 1n \sum_{k=1}^{n-1}|\xi^{(n-1)}_k|^{p} = \int_{\R} |x|^p \mu(\dint x)
\end{align*}
and
\begin{align*}
&\liminf_{n\to\infty} \frac 1{n^2} \sum_{1\leq i<j<n} \log |a_j^{(n-1)}-b_i^{(n-1)}| \geq \frac 12
\int_{\R}\int_{\R} \log|x-y| \,\mu(\dint x)\, \mu(\dint y),
\end{align*}
we obtain
\begin{align*}
& \liminf_{n\to\infty} \frac{1}{n^2}\log \Pro \left[\nu'_{n-1}\in\mathcal O_{\varepsilon,d}(\mu), \;
\left|\int_{\R}|y|^p\,\nu'_{n-1}(\dint y)-\int_{\R}|y|^p\,\mu(\dint y)\right| < \delta,
\; X_{n,n}\in D_n \right] \\
& \geq -B-m + \frac \beta 2 \int_{\R}\int_{\R} \log|x-y| \,\mu(\dint x)\, \mu(\dint y),
\end{align*}
which is the required estimate.
\end{proof}

Next we shall lift the previous lemma to arbitrary measures $\mu\in\mathcal M(\R)$ and prove Lemma~\ref{lem:lower_bound} stating the lower bound in the weak LDP for the pair $\left(\nu_n, \int_\R |x|^p \nu_n(\dint x)\right)$.

\begin{lemma}
Let $\mu\in\mathcal M(\R)$ and $m\in (0,+\infty)$ be such that $m\geq\int_\R |x|^p \mu(\dint x)$.
Let $\varepsilon,\delta>0$ and $d\in\N$, $f_1,\ldots,f_d\in \mathcal C_b(\R)$. Then
\begin{align*}
\liminf_{n\to\infty} \frac{1}{n^2} \log &\,\Pro \left[\nu_{n}\in\mathcal O_{\varepsilon,d}(\mu), \int_{\R}|y|^p\,\nu_{n}(\dint y)\in (m - \delta, m+\delta) \right] \\
& \geq \frac \beta 2 \int_{\R}\int_{\R} \log|x-y| \,\mu(\dint x)\, \mu(\dint y) -B-m.
\end{align*}
\end{lemma}
\begin{proof}
Following the argumentation of \cite[Lemma 5.4.6]{hiai_petz} there is a sequence $(\mu_k)_{k\in\N}$ of probability measures on $\R$ such that, for all $k\in\N$,
\begin{itemize}
\item[(i)] $\mu_k$ is supported on some interval $[a_k,b_k]$ and has continuous Lebesgue density $h_k$ on this interval such that $\inf_{x\in[a_k,b_k]}h_k(x)>0$,
\item[(ii)] $\mu_k$ converges to $\mu$ weakly, as $k\to\infty$,
\item[(iii)] for $k\to\infty$,
\[
\int_{\R}|x|^p\,\mu_k(\dint x) \to \int_{\R}|x|^p \,\mu(\dint x).
\]
\end{itemize}
As a consequence of (ii) and (iii), for all sufficiently large $k\in\N$,
\[
\mathcal O_{\eps/3,d}(\mu_k) \subseteq \mathcal O_{\varepsilon,d}(\mu).
\]
By (iii), we can construct a sequence $(m_k)_{k\in\N}$ such that $\lim_{k\to\infty} m_k = m$ and $m_k > \int_{\R}|x|^p\,\mu_k(\dint x)$ for all $k\in\N$. Indeed, we may simply take
\[
m_k := m - \Big(\int_{\R}|x|^p\,\mu(\dint x)-\int_{\R}|x|^p\,\mu_k(\dint x)\Big)+\frac{1}{k},\qquad k\in \N.
\]  
Therefore, for large enough $k\in\N$,
\begin{align*}
& \Pro \left[\nu_{n}\in\mathcal O_{\varepsilon,d}(\mu), \int_{\R}|y|^p\,\nu_{n}(\dint y)\in (m - \delta, m+\delta) \right] \\
&\geq \Pro \left[\nu_{n}\in\mathcal O_{\varepsilon/3,d}(\mu_k), \int_{\R}|y|^p\,\nu_{n}(\dint y)\in \left(m_k - \frac \delta2, m_k+\frac \delta2\right) \right]\\
& \geq \Pro \left[\nu'_{n-1}\in\mathcal O_{\eps/9,d}(\mu_k), \; \left|\int_{\R}|y|^p\,\nu'_{n-1}(\dint y)-\int_{\R}|y|^p\,\mu_k(\dint y)\right| < \frac{\delta}{6},\; X_{n,n}\in D_{n,k} \right],
\end{align*}
where the last inequality follows from Lemma~\ref{lem:smaller_event} and $D_{n,k}$ is defined in the same way as $D_n$, see~\eqref{eq:def_D_n},  but with $m$ replaced by $m_k$.
We can now apply Lemma \ref{lem:lower bound pair compactly supported} with $\mu$ replaced by $\mu_k$ there and obtain, for every sufficiently large fixed $k\in\N$,
\[
\liminf_{n\to\infty} \frac{1}{n^2} \log \Pro \left[\nu_{n}\in\mathcal O_{\varepsilon,d}(\mu), \int_{\R}|y|^p\,\nu_{n}(\dint y)\in (m - \delta, m+\delta)\right] \geq -\mathscr I_1(\mu_k,m_k).
\]
It is now left to show that
\[
\limsup_{k\to\infty} \mathscr I_1(\mu_k,m_k) \leq \mathscr I_1(\mu,m).
\]
But this follows from the upper semi-continuity of the free entropy in the following form: under conditions (ii) and (iii) we have
\[
\limsup_{k\to\infty} \int_{\R}\int_{\R}\log|x-y|\,\mu_k(\dint x)\,\mu_k(\dint y) \leq \int_{\R}\int_{\R}\log|x-y|\,\mu(\dint x)\,\mu(\dint y).
\]
This standard fact can be verified by the same argument as on p.~214 of~\cite{hiai_petz}.
\end{proof}

\subsubsection{The upper bound}
Again we recall that $m \geq \int_{\R}|x|^p\,\mu(\dint x)$.
To obtain the upper bound, we follow a classical idea and consider an appropriate kernel function together with its truncated version (see, e.g., \cite{hiai_petz_wishart,HP2000,hiai_petz}).

\begin{lemma}\label{lem:upper_bound}
Let $\mu\in\mathcal M(\R)$ and $m\in (0,+\infty)$ be such that $m\geq\int_\R |x|^p \mu(\dint x)$.
Let $\varepsilon,\delta>0$ and $d\in\N$, $f_1,\ldots,f_d\in \mathcal C_b(\R)$. Then
\begin{align*}
\lim_{\varepsilon\downarrow 0, \delta\downarrow 0}
\limsup_{n\to\infty}
\frac{1}{n^2} \log &\Pro \left[\nu_{n}\in\mathcal O_{\varepsilon,d}(\mu), \int_{\R} |y|^p\,\nu_{n} (\dint y) \in (m-\delta, m+\delta)\right] \\
&\leq \frac \beta 2 \int_{\R}\int_{\R} \log|x-y| \,\mu(\dint x)\, \mu(\dint y)-B-m.
\end{align*}
\end{lemma}
\begin{proof}
For some $\gamma>0$ define the following weighted logarithmic kernel
\[
F(x,y; \gamma) := 
\begin{cases}
- \frac \beta 2 \log|x-y| + \gamma\frac{|x|^p+|y|^p}{2} &: x\neq y \\
+\infty &: x=y,
\end{cases}
\]
and for $\alpha>0$ its $\alpha$-truncated version by
\[
F_{\alpha}(x,y;\gamma) := \min\{F(x,y;\gamma),\alpha \}.
\]
We observe that $F_\alpha(x,x; \gamma)=\alpha$. For $t=(t_1,\ldots,t_n)\in\R^n$, let us write $\mu_t= \frac 1n \sum_{i=1}^n \delta_{t_i}$ for the empirical measure of $t$ and denote by $H$ the set
\[
H := \left\{t\in\R^n\,:\, \mu_t \in\mathcal O_{\varepsilon,d}(\mu),\,\frac{1}{n}\sum_{i=1}^n|t_i|^p \in(m-\delta,m+\delta) \right\},
\]
where we recall that the definition of $\mathcal O_{\varepsilon,d}(\mu)$ depends on the functions $f_1,\ldots,f_d$.
We have
\begin{align*}
\lefteqn{\Pro\left[\nu_{n}\in\mathcal O_{\varepsilon,d}(\mu), \int_{\R}|y|^p\,\nu_{n}(\dint y) \in (m-\delta, m+\delta)\right]}\\
&=
\frac{1}{C_{n,\beta,p}} \int_{H}
\exp\left\{- n \sum_{i=1}^n |t_i|^p\right\}
\exp\left\{\beta \sum_{1\leq i <j \leq n} \log |t_j-t_i|\right\}\,\dint(t_1, \ldots , t_n) \\
& =
\frac{1}{C_{n,\beta,p}} \int_{H} \exp\left\{-\sum_{i=1}^n |t_i|^p(n+\gamma-\gamma n)\right\}
\exp\left\{-2\sum_{1\leq i<j\leq n}F(t_i,t_j;\gamma)\right\}\,\dint(t_1, \ldots, t_n) \cr
& \leq
\frac{1}{C_{n,\beta,p}} \int_{H} e^{(m-\delta)n(\gamma n -\gamma - n)}
\exp\left\{-n^2\int_{\R}\int_{\R} F_\alpha(x,y;\gamma) \,\mu_t(\dint x)\,\mu_t(\dint y) + n\alpha\right\} \,\dint(t_1, \ldots, t_n) \cr
& \leq
\frac{1}{C_{n,\beta,p}}e^{(m-\delta)n^2(\gamma -1)-(m-\delta)n\gamma}\exp\left\{-n^2\inf\limits_{\mu'\in\mathcal O_{\varepsilon,d}(\mu)}\int_{\R}\int_{\R} F_\alpha(x,y;\gamma)\,\mu'(\dint x)\,\mu'(\dint y) + n\alpha\right\},
\end{align*}
where in the second step we used that
\begin{align*}
-n\sum_{i=1}^n|t_i|^p&+\beta\sum_{1\leq i<j\leq n}\log|t_i-t_j|\\
&=-n\sum_{i=1}^n|t_i|^p-2\sum_{1\leq i<j\leq n}F(t_i,t_j;\gamma)+\gamma\sum_{1\leq i<j\leq n}(|t_i|^p+|t_j|^p)\\
&=-n\sum_{i=1}^n|t_i|^p-2\sum_{1\leq i<j\leq n}F(t_i,t_j;\gamma)+\gamma\sum_{i=1}^{n-1}\Big((n-i)|t_i|^p+\sum_{j=i+1}^n|t_j|^p\Big)\\
&=-n\sum_{i=1}^n|t_i|^p-2\sum_{1\leq i<j\leq n}F(t_i,t_j;\gamma)+\gamma(n-1)\sum_{i=1}^n|t_i|^p\\
&=-\sum_{i=1}^n|t_i|^p(n+\gamma-\gamma n)-2\sum_{1\leq i<j\leq n}F(t_i,t_j;\gamma).
\end{align*}
Taking the logarithm and dividing by $n^2$ yields
\begin{align*}
& \frac{1}{n^2} \log \Pro\left[\nu_{n}\in\mathcal O_{\varepsilon,d}(\mu), \int_{\R}|y|^p\,\nu_{n}(\dint y) \in (m-\delta, m+\delta)\right] \cr
& \leq -\frac{1}{n^2}\log C_{n,\beta,p} + (m-\delta)(\gamma-1) - \frac{\gamma(m-\delta)}{n} - \inf\limits_{\mu'\in\mathcal O_{\varepsilon,d}(\mu)}\int_{\R}\int_{\R} F_\alpha(x,y;\gamma)\,\mu'(\dint x)\,\mu'(\dint y) + \frac{\alpha}{n}.
\end{align*}
Letting $n\to\infty$, we arrive at
\begin{multline*}
\limsup_{n\to\infty} \frac{1}{n^2} \log \Pro\left[
\nu_{n}\in\mathcal O_{\varepsilon,d}(\mu), \int_{\R}|y|^p\,\nu_{n}(\dint y) \in (m-\delta, m+\delta)\right] \cr
\leq -B + (m-\delta)(\gamma-1) - \inf\limits_{\mu'\in\mathcal O_{\varepsilon,d}(\mu)}\int_{\R}\int_{\R} F_\alpha(x,y;\gamma)\,\mu'(\dint x)\,\mu'(\dint y).
\end{multline*}
The function $(x,y) \mapsto F_\alpha(x,y;\gamma)$ is continuous and bounded on $\R^2$. It follows that  the functional
\[
\mathcal M(\R)\to\R,\quad\mu'\mapsto \int_{\R}\int_{\R} F_\alpha(x,y;\gamma) \mu'(\dint x)\,\mu'(\dint y)
\]
is weakly continuous.
Now, if we take $\varepsilon\downarrow 0$ as well as $\delta\downarrow0$, then
\begin{multline*}
\lim_{\varepsilon\downarrow 0, \delta\downarrow 0}\limsup_{n\to\infty}
\frac{1}{n^2} \log \Pro\left[\nu_{n}\in\mathcal O_{\varepsilon,d}(\mu), \int_{\R}|y|^p\,\nu_{n}(\dint y) \in (m-\delta, m+\delta)  \right] \cr
\leq -B + m(\gamma-1) - \int_{\R}\int_{\R} F_\alpha(x,y;\gamma)\,\mu(\dint x)\,\mu(\dint y).
\end{multline*}
This holds for all $\alpha>0$ and $\gamma>0$. Fix any $\alpha>0$.
As $\gamma\downarrow 0$, we have
$$
F_{\alpha}(x,y;\gamma) \downarrow \min\left\{- \frac \beta 2 \log|x-y|,\alpha\right\}
$$
for every $(x,y)\in \R^2$. Moreover, we have the bound $F_{\alpha}(x,y;\gamma)\leq \alpha$.  The monotone convergence theorem, applied to the functions $(-F_{\alpha}(x,y;\gamma))_{0<\gamma<1}$, yields
\begin{multline*}
\lim_{\varepsilon\downarrow 0, \delta\downarrow 0}\limsup_{n\to\infty}
\frac{1}{n^2} \log \Pro\left[\nu_{n}\in\mathcal O_{\varepsilon,d}(\mu), \int_{\R}|y|^p\,\nu_{n}(\dint y) \in (m-\delta, m+\delta)  \right] \cr
\leq -B - m - \int_{\R}\int_{\R} \min\left\{- \frac \beta 2 \log|x-y|,\alpha\right\}\,\mu(\dint x)\,\mu(\dint y).
\end{multline*}
Next we let $\alpha\to +\infty$, use the monotone convergence theorem this time for the double integral over the set $\{(x,y)\in\R^2: |x-y| \leq 1\}$, and observe that the integral over $\{(x,y)\in\R^2: |x-y| \geq 1\}$ does not depend on $\alpha$,  to get
\begin{multline*}
\lim_{\varepsilon\downarrow 0, \delta\downarrow 0}\limsup_{n\to\infty}
\frac{1}{n^2} \log \Pro\left[\nu_{n}\in\mathcal O_{\varepsilon,d}(\mu), \int_{\R}|y|^p\,\nu_{n}(\dint y) \in (m-\delta, m+\delta)  \right] \cr
\leq -B - m + \frac \beta 2 \int_{\R}\int_{\R}   \log|x-y| \,\mu(\dint x)\,\mu(\dint y).
\end{multline*}
This completes the proof of the lemma.
\end{proof}

\subsection{Exponential tightness}
Having established the weak LDP for the sequence of pairs $(\nu_n,\int_\R |x|^p \nu_n(\dint x))$, we proceed to the exponential tightness. The next lemma is the last missing part of  the proof of Theorem~\ref{thm:ldp empirical pair}.
\begin{lemma}\label{lem:exp_tight}
The sequence of random elements $(\nu_n,\int_\R |x|^p \nu_n(\dint x))$, $n\in\N$, is exponentially tight on $\mathcal M(\R) \times [0,+\infty)$.
\end{lemma}
\begin{proof}
For each $A>0$ the set $K_A:= \{\nu \in \mathcal M(\R): \int_\R |x|^p \nu(\dint x) \leq A\}$ is compact in $\mathcal M(\R)$. Indeed, given some $\varepsilon>0$, putting $B:=(A/\eps)^{1/p}$ and using Markov's inequality, we find that
$$
\nu(\{x\in\R:|x|>B\}) \leq {\int_{\R}|x|^p\nu(\dint x) \over B^p}\leq {A\over B^p} = \eps.
$$
Thus,
\[
\sup_{\nu\in K_A}\nu([-B,B]^c) \leq \varepsilon,
\] 
which shows that the family $K_A$ is tight. The weak compactness is finally a consequence of Prohorov's theorem (see \cite[Theorem 14.3]{Kallenberg}). Hence, the set $K_A^*:=  K_A \times [0,A]$
is compact in $\mathcal M(\R)\times [0,+\infty)$. It suffices to prove that
$$
\lim_{A\to +\infty} \limsup_{n\to\infty} \frac 1{n^2} \log \Pro\left[\left(\nu_n,\int_\R |x|^p \nu_n(\dint x)\right)\notin K_A^*\right] = -\infty,
$$
which is equivalent to
\begin{equation}\label{eq:tightness_proof_1}
\lim_{A\to +\infty} \limsup_{n\to\infty} \frac 1{n^2} \log \Pro\left[\int_\R |x|^p \nu_n(\dint x) >A \right] = -\infty.
\end{equation}
However, the latter property has been established in the proof of~\cite[Lemma~5.4.8]{hiai_petz}.
\end{proof}

\begin{rmk}\label{rem:LDP_on_(0,infty)}
Theorem~\ref{thm:ldp empirical pair} continues to hold if the space $\mathcal M(\R) \times [0,+\infty)$ is replaced by $\mathcal M(\R) \times (0,+\infty)$. Clearly, we can consider $(\nu_n, \int_\R |x|^p \nu_n(\dint x))$ as an element of the latter space since the probability of the event $\{\int_\R |x|^p \nu_n(\dint x)=0\}$ is zero.  The proof of the weak LDP does not change. However, in the proof of exponential tightness the set $K_A\times (0,A]$ is not compact and has to be replaced by $K_A\times [A^{-1},A]$. It remains to check that
\begin{equation}\label{eq:tightness_proof_4}
\lim_{A\to +\infty} \limsup_{n\to\infty} \frac 1{n^2} \log \Pro\left[\int_\R |x|^p \nu_n(\dint x) < \frac 1A \right] = -\infty.
\end{equation}
Using the formula for the joint density of $(X_{1,n},\ldots,X_{n,n})$ given in~\eqref{eq:distribution of (X_1,...,X_n)}, we can write
\begin{align*}
&\Pro\left[\int_\R |x|^p \nu_n(\dint x) < \frac 1A \right]=
\Pro\left[\sum_{i=1}^n |X_{i,n}|^p  < \frac nA \right]\\
&=
\frac 1 {C_{n,\beta,p}} \int_{\R^n} \left(e^{-n \sum_{i=1}^n |x_i|^p} \prod_{1\leq i < j \leq n} |x_j-x_i|^\beta\right) \ind_{\left\{\sum_{i=1}^n |x_i|^p < \frac n A\right\}} \,\dint (x_1, \ldots,  x_n)\\
&\leq
\frac 1 {C_{n,\beta,p}} \int_{\R^n} \left(\prod_{1\leq i < j \leq n} |x_j-x_i|^\beta\right) \ind_{\left\{\sum_{i=1}^n |x_i|^p < \frac n A\right\}} \,\dint( x_1 ,\ldots , x_n)\\
&\leq
\frac 1 {C_{n,\beta,p}}  \left(\frac n A \right)^{\frac {\beta n(n-1)}{2p} + \frac {n}{p}}\int_{\R^n} \left(\prod_{1\leq i < j \leq n} |\lambda_j-\lambda_i|^\beta\right) \ind_{\left\{\sum_{i=1}^n |\lambda_i|^p < 1\right\}} \,\dint (\lambda_1, \ldots , \lambda_n),
\end{align*}
where we used the change of variables $x_i = (n/A)^{1/p} \lambda_i$, $1\leq i \leq n$. The asymptotic behavior of the integral on the right-hand side has been determined in~\cite[Lemma~3.9]{KPT2018a}, where this integral was denoted by $I_{n,\beta,p}$. For some constant $\Delta(p)>0$ (whose explicit value can be found in~\cite[Theorem~3.1]{KPT2018a}), we have
\begin{align*}
\Pro\left[\int_\R |x|^p \nu_n(\dint x) < \frac 1A \right]
&\leq
\frac 1 {C_{n,\beta,p}}  \left(\frac n A \right)^{\frac {\beta n(n-1)}{2p} + \frac {n}{p}}
n^{-\frac{\beta n^2}{2 p}} \big(\Delta^{\beta}(p)(1+o(1))\big)^{\frac{n^2}{2}}\\
&\leq
e^{O(1) n^2} \left(\frac 1 A \right)^{\frac {\beta n(n-1)}{2p} + \frac {n}{p}},
\end{align*}
where the $O(1)$-constant does not depend on $A$. It follows that
$$
\limsup_{n\to\infty}\frac 1 {n^2} \log \Pro\left[\int_\R |x|^p \nu_n(\dint x) < \frac 1A \right]
\leq
O(1)  - \frac {\beta \log A}{2p},
$$
which implies \eqref{eq:tightness_proof_4}.
\end{rmk}
\section{Step 2 -- LDP for the cone measure}

We shall now use the contraction principle (see Proposition~\ref{prop:contraction principle})
to prove the part of Theorem~\ref{theo:main} relating to the cone measure.
\begin{proposition}\label{prop:LDP_cone}
For every $n\in\N$, let $Z_n$ be a random matrix sampled according to the cone measure on $\Sph_{p,\beta}^{n-1}$. Then the random probability measure
\[
\mu_n = \frac{1}{n}\sum_{i=1}^n \delta_{n^{1/p}\lambda_i(Z_n)}
\]
satisfies an LDP on $\mathcal M(\R)$ with speed $n^2$ and good rate function $\mathscr I:\mathcal M(\R) \to [0,+\infty]$  given by~\eqref{eq:J_def_rate_funct}.
\end{proposition}
\begin{proof}
Recall from~\eqref{eq:schechtman_zinn_cone} the distributional equality
$$
\mu_n = \frac{1}{n}\sum_{i=1}^n \delta_{n^{1/p}\lambda_i(Z_n)} \eqdistr \frac{1}{n}\sum_{i=1}^n \delta_{\frac{n^{1/p}X_{i,n}}{\|X_n\|_p}}.
$$
We shall prove an LDP for the right-hand side. To this end,  we shall apply contraction principle to the LDP obtained in Theorem~\ref{thm:ldp empirical pair}.  Let us define the function
\begin{equation}\label{eq:definition F_p}
 F_p:\cM(\R)\times(0,+\infty)\to\cM(\R),\quad F_p(\mu,c)(A):=\mu(c^{1/p}A),
\end{equation}
where $A$ is any Borel subset of $\R$. Note that the value $c=0$ is excluded.  It follows from \cite[Lemma 3.1]{KimRamanan} that this function is continuous.  In the following, it will be convenient to use the notation
$$
m_p(\nu) := \int_{\R}|x|^p \nu(\dint x), \qquad \nu\in \mathcal M(\R).
$$
Recall that $\nu_n := \frac 1n \sum_{i=1}^n \delta_{X_{i,n}}$ and observe that, for every $n\in\N$ and all Borel sets $A\subseteq\R$,
\[
F_p\left(\nu_n,m_p(\nu_n)\right)(A)
=
\nu_n\left(m_p^{1/p}(\nu_n) A\right)
=
\frac{1}{n}\sum_{i=1}^n \delta_{X_{i,n}}\left(n^{-1/p} \|X_n\|_p A\right)
=
\frac{1}{n}\sum_{i=1}^n \delta_{\frac{n^{1/p}X_{i,n}}{\|X_n\|_p}}\left(A\right).
\]
Since the mapping $F_p$ defined in \eqref{eq:definition F_p} is continuous, it follows from the contraction principle (see Proposition~\ref{prop:contraction principle})  in combination with Theorem~\ref{thm:ldp empirical pair} above (see also Remark~\ref{rem:LDP_on_(0,infty)}) that the sequence of random elements
\[
F_p\left(\nu_n,m_p(\nu_n)\right) = \frac{1}{n}\sum_{i=1}^n \delta_{\frac{n^{1/p}X_{i,n}}{\|X_n\|_p}},\qquad n\in\N,
\]
satisfies an LDP with speed $n^2$ and good rate function $\mathscr I_2:\mathcal M(\R) \to [0,+\infty]$ given by
\[
\mathscr I_2(\mu) = \inf_{(\nu,m):\, F_p(\nu,m)=\mu}\mathscr I_1(\nu,m) = \inf_{(\nu,m):\,\nu(m^{1/p}\cdot) = \mu(\cdot)}\mathscr I_1(\nu,m),
\]
where $\mathscr I_1:\mathcal M(\R) \times (0,+\infty)\to [0,+\infty]$ is the rate function from Theorem~\ref{thm:ldp empirical pair}. It remains to check that $\mathscr I_2 (\mu) = \mathscr I(\mu)$ for all $\mu\in\mathcal M(\R)$.

\vspace*{2mm}
\noindent
\textit{Case 1.}
Let $\mu\in \mathcal M(\R)$ be such that  $m_p(\mu)>1$. If $(\nu,m)$ is such that $\mu (\cdot) = \nu(m^{1/p}\cdot)$, then
\[
1< m_p(\mu) = \int_{\R} |x|^p \,\mu(\dint x) = \int_{\R}|m^{-1/p}y|^p\, \nu(\dint y),
\]
which implies $m< \int_{\R}|y|^p \,\nu(\dint y)$.
Consequently, $\mathscr I_1(\nu,m) = +\infty$, which shows that in this case $\mathscr I_2(\mu) = +\infty= \mathscr I(\mu)$.

\vspace*{2mm}
\noindent
\textit{Case 2.}
Now assume that $\mu\in\mathcal M(\R)$ is such that $m_p(\mu)\leq 1$. If $(\nu,m)$ is such that  $\mu (\cdot) = \nu(m^{1/p}\cdot)$, then
$
m \geq \int_{\R}|y|^p \,\nu(\dint y)
$.
We obtain
\begin{align*}
\mathscr I(\mu) & = \inf_{(\nu,m):\, \nu(m^{1/p}\cdot) = \mu(\cdot)}\mathscr I_1(\nu,m) \cr
& = \inf_{(\nu,m):\,\nu(m^{1/p}\cdot) = \mu(\cdot)}\left(- \frac \beta2 \int_{\R}\int_{\R} \log|x-y| \,\nu(\dint x)\, \nu(\dint y)  + m + B\right) \cr
& = \inf_{(\nu,m):\,\nu(m^{1/p}\cdot) = \mu(\cdot)}\left(- \frac \beta 2 \int_{\R}\int_{\R} \log\frac{|x-y|}{m^{1/p}} \,\nu(\dint x)\, \nu(\dint y) - \frac{\beta}{2p}\log m + m + B \right) \cr
& = \inf_{(\nu,m):\,\nu(m^{1/p}\cdot) = \mu(\cdot)}\left(- \frac \beta 2 \int_{\R}\int_{\R} \log|x-y| \,\mu(\dint x)\, \mu(\dint y) - \frac{\beta}{2p}\log m + m + B\right) \cr
& = - \frac \beta 2 \int_{\R}\int_{\R} \log|x-y| \,\mu(\dint x)\, \mu(\dint y) - \frac{\beta}{2p}\log \frac{\beta}{2p} + \frac{\beta}{2p} + B.
\end{align*}
Recalling the formula for $B$ given in~\eqref{eq:B_const_formula}, we can compute the constant term as follows:
\begin{align*}
B - \frac{\beta}{2p}\log \frac{\beta}{2p} + \frac{\beta}{2p}
&=
\left(\frac {\beta}{2p} \log \left(\frac{\beta \sqrt \pi \Gamma\left(\frac p2\right)}{2 \Gamma\left(\frac{p+1}{2}\right)}\right) - \frac \beta 2 \log 2 - \frac {3\beta}{4p}\right) -\frac{\beta}{2p}\log \frac{\beta}{2p} + \frac{\beta}{2p} \\
&=
\frac{\beta}{2p} \log\left(\frac{\sqrt{\pi}p \Gamma(\frac{p}{2})}{2^p\sqrt{e}\Gamma(\frac{p+1}{2})}\right),
\end{align*}
which completes the proof that $\mathscr I_2(\mu) = \mathscr I(\mu)$.
\end{proof}

\section{Step 3 - LDP for the uniform distribution on the ball}

We shall now apply once more the contraction principle to prove the remaining part of Theorem~\ref{theo:main}. We prepare the proof with the following auxiliary result on large deviations for sequences of beta random variables. We fix $a,b>0$ and recall that a random variable $\xi$ is beta distributed with parameters $a$ and $b$, we write $\xi\sim \beta_{a,b}$, provided that its Lebesgue density is given by
$$
x\mapsto {1\over B(a,b)}x^{a-1}(1-x)^{b-1},\qquad x\in[0,1]\,,
$$
where $B(a,b)={\Gamma(a)\Gamma(b)\over\Gamma(a+b)}$ is Euler's beta function. The next result is a slightly more general version of \cite[Lemma 4.1]{APT2018}. The proof is almost literally the same and for this reason omitted.

\begin{lemma}\label{lem:LDP-Beta}
Fix some $s>0$ and let $(a_n)_{n\in\N}$ and $(b_n)_{n\in\N}$ be positive sequences such that the limits
$$
\lim\limits_{n\to\infty}{a_n\over n^s}=a\in[0,\infty)\qquad\text{and}\qquad\lim\limits_{n\to\infty}{b_n\over n^s}=b\in[0,\infty)
$$
exist are not equal to zero at the same time. For each $n\in\N$ assume that  $\xi_n$ is beta distributed with parameters $a_n$ and $b_n$. Then the sequence $(\xi_n)_{n\in\N}$ satisfies a LDP with speed $n^s$ and rate function $\mathscr{I}_{a,b}$ given by
$$
\mathscr{I}_{a,b}(y) = \begin{cases}
-a\log{y\over a}-b\log{1-y\over b}-(a+b)\log(a+b) &: y\in(0,1)\text{ and }a>0,b>0\\
-a\log{y} &: y\in(0,1]\text{ and }a>0,b=0\\
-b\log{(1-y)} &: y\in[0,1)\text{ and }a=0,b>0\,.
\end{cases}
$$
\end{lemma}

We are now prepared to present the last step in the proof of Theorem~\ref{theo:main}.

\begin{proposition}
For every $n\in\N$, let $Z_n$ be a random matrix sampled uniformly from $\B_{p,\beta}^{n}$. Then the random probability measure
\[
\mu_n = \frac{1}{n}\sum_{i=1}^n \delta_{n^{1/p}\lambda_i(Z_n)}
\]
satisfies an LDP on $\mathcal M(\R)$ with speed $n^2$ and good rate function $\mathscr I:\mathcal M(\R) \to [0,+\infty]$  given by~\eqref{eq:J_def_rate_funct}.
\end{proposition}
\begin{proof}
Recall from~\eqref{eq:schechtman_zinn_uniform} the representation
\begin{equation*}
\mu_n
=
\frac 1n \sum_{i=1}^n \delta_{n^{1/p}\lambda_i(Z_n)}
\eqdistr
\frac 1n \sum_{i=1}^n \delta_{U^{1 /\ell}{X_{i,n}\over \|X_n\|_p}},
\end{equation*}
where $\ell=\ell(n,\beta) = \beta\frac{n(n-1)}{2} + \beta n$, $X_n=(X_{1,n},\dots,X_{n,n})$ has joint Lebesgue density given in~\eqref{eq:distribution of (X_1,...,X_n)} and $U$ is independent from $X_n$ and uniformly distributed on $[0,1]$.
The ultimate goal is to prove an LDP for the sequence
\[
\frac{1}{n}\sum_{i=1}^n \delta_{U^{1/\ell}n^{1/p}\frac{X_{i,n}}{\|X\|_p}}.
\]
Note that for each $n\in\N$ we have that $U^{1/\ell}\sim\beta_{\ell,1}$. Thus, Lemma \ref{lem:LDP-Beta} with $a_n=\ell$, $b_n=0$, $a=\beta/2$, $b=0$ and $s=2$ implies that the sequence $(U^{1/\ell})_{n\in\N}$ satisfies an LDP with speed $n^2$ and rate function
\[
\mathscr I_{\beta/2,0}(x) = \begin{cases}
-\frac{\beta}{2}\log x &\,: x\in(0,1] \\
+\infty &\,: \text{otherwise}.
\end{cases}
\]
In the next step towards the final LDP, we observe that as a direct consequence of Proposition~\ref{JointRateFunction}, Proposition~\ref{prop:LDP_cone} and the previous LDP for $(U^{1/\ell})_{n\in\N}$, the sequence of pairs
\[
\left(\left(\frac{1}{n}\sum_{i=1}^n \delta_{\frac{n^{1/p}X_{i,n}}{\|X_n\|_p}},U^{1/\ell}\right)\right)_{n\in\N}
\]
satisfies an LDP with speed $n^2$ and rate function $\mathscr I_4:\mathcal M(\R)\times (0,+\infty) \to[0,+\infty]$ given by
\begin{align*}
\mathscr I_4(\nu,z) & = \mathscr I(\nu) + \mathscr I_{\beta/2,0}(z) \cr
& = - \frac \beta 2 \int_{\R}\int_{\R} \log|x-y| \,\nu(\dint x)\, \nu(\dint y) + \frac{\beta}{2p} \log\left(\frac{\sqrt{\pi}p \Gamma(\frac{p}{2})}{2^p\sqrt{e}\Gamma(\frac{p+1}{2})}\right) - \frac{\beta}{2}\log z.
\end{align*}
This time we shall apply the contraction principle with a continuous function similar to the one defined in \eqref{eq:definition F_p} but for $p=-1$, i.e., with
\begin{equation*}
 F:\cM(\R)\times(0,+\infty)\to\cM(\R),\quad F(\mu,c)(A)=\mu\left(\frac{1}{c}A\right),
\end{equation*}
where $A$ is any Borel subset of $\R$. This map is continuous.  Note that, for any Borel set $A\subseteq \R$,
\begin{equation}\label{eq:epm_measures_techn}
F\left(\left(\frac{1}{n}\sum_{i=1}^n \delta_{\frac{n^{1/p}X_{i,n}}{\|X_n\|_p}},U^{1/\ell}\right) \right)(A) = \frac{1}{n}\sum_{i=1}^n \delta_{\frac{n^{1/p}X_{i,n}}{\|X_n\|_p}}\left(U^{-1/\ell}A\right) =
\frac{1}{n}\sum_{i=1}^n \delta_{U^{1/\ell}\frac{n^{1/p}X_{i,n}}{\|X_n\|_p}}(A).
\end{equation}
The contraction principle (see Proposition \ref{prop:contraction principle}) shows that the empirical measures in~\eqref{eq:epm_measures_techn} satisfy an LDP on $\mathcal M(\R)$ with speed $n^2$ and rate function
$$
\mathscr I_5(\mu)  = \inf_{(\nu,z):\, F(\nu,z)=\mu}\mathscr I_4(\nu,z)
=
\inf_{(\nu,z):\, F(\nu,z)=\mu}
\left(\mathscr I(\nu) - \frac \beta 2 \log z\right).
$$
Here and in the following, $z$ is restricted to the interval $(0,1]$.
It remains to check that $\mathscr I_5(\mu) = \mathscr I(\mu)$ for all $\mu\in \mathcal M(\R)$.

\vspace*{2mm}
\noindent
\textit{Case 1.}
Let $\mu\in \mathcal M(\R)$ be such that  $m_p(\mu)>1$. If $(\nu,z)$ is such that $\mu (\cdot) = \nu(z^{-1}\cdot)$, then $m_p(\nu)= z^{-p} m_p(\mu) >1$. Consequently, $\mathscr I(\nu) = +\infty$ and it follows that in this case $\mathscr I_5(\mu) = +\infty= \mathscr I(\mu)$.

\vspace*{2mm}
\noindent
\textit{Case 2.}
Let $\mu\in \mathcal M(\R)$ be such that  $m_p(\mu)\leq 1$. If $(\nu,z)$ is such that $\mu (\cdot) = \nu(z^{-1}\cdot)$, then $m_p(\nu)= z^{-p} m_p(\mu)$ and we have two possibilities. If $m_p(\nu)>1$, then $\mathscr I(\nu) - \frac{\beta}{2} \log z =+\infty$. So we let in the following $m_p(\nu)\leq 1$, which means that $z$ is restricted to the non-empty interval $[m_p^{1/p}(\mu), 1]$. Then,
\begin{align*}
\mathscr I(\mu) & = \inf_{(\nu,z):\, F(\nu,z)=\mu}\mathscr I_4(\nu,z) \cr
& = \inf_{(\nu,z):\, \nu(z^{-1}\cdot)=\mu(\cdot)}\left( - \frac \beta 2 \int_{\R}\int_{\R} \log|x-y| \,\nu(\dint x)\, \nu(\dint y) \right.\\
&\qquad\qquad\qquad\qquad\qquad\qquad\left.+\frac{\beta}{2p} \log\left(\frac{\sqrt{\pi}p \Gamma(\frac{p}{2})}{2^p\sqrt{e}\Gamma(\frac{p+1}{2})}\right) - \frac{\beta}{2}\log z\right) \cr
& = \inf_{(\nu,z):\, \nu(z^{-1}\cdot)=\mu(\cdot)}\left(\frac \beta 2 \log z - \frac \beta 2 \int_{\R}\int_{\R} \log|x-y| \,\mu(\dint x)\, \mu(\dint y)\right.\\
&\qquad\qquad\qquad\qquad\qquad\qquad\left.+\frac{\beta}{2p} \log\left(\frac{\sqrt{\pi}p \Gamma(\frac{p}{2})}{2^p\sqrt{e}\Gamma(\frac{p+1}{2})}\right) - \frac{\beta}{2}\log z\right) \cr
& = -\frac \beta 2 \int_{\R}\int_{\R} \log|x-y| \,\mu(\dint x)\, \mu(\dint y) + \frac{\beta}{2p} \log\left(\frac{\sqrt{\pi}p \Gamma(\frac{p}{2})}{2^p\sqrt{e}\Gamma(\frac{p+1}{2})}\right),
\end{align*}
which shows that again $\mathscr I_5(\mu)= \mathscr I(\mu)$.
\end{proof}

\section {Proof of Corollary~\ref{cor:LLN}} \label{sec:proof_LLN}

By \cite[Proposition 5.3.4]{hiai_petz}, the probability measure $\mu_\infty^{(p)}$ is the only minimizer of $\mathscr I$ and it holds that $\mathscr I(\mu_\infty^{(p)})= 0$.

Let $d(\,\cdot\,,\, \cdot\,)$ be any metric on $\mathcal M(\R)$ metrizing the topology of weak convergence. Then, for every $\eps>0$,
$$
\limsup_{n\to\infty} \frac 1{n^2} \log  \Pro\left[d(\mu_n, \mu_\infty^{(p)}) \geq \eps \right]
\leq -\inf_{\tau\in \mathcal M(\R):\, d(\tau,\mu_\infty^{(p)})\geq \eps} \mathscr I(\tau).
$$
We claim that the infimum on the right-hand side is strictly negative.  Indeed, if it were $0$, we could find a sequence $(\tau_k)_{k\in\N}$ such that $d(\tau_k,\mu_\infty^{(p)})\geq \eps$ for all $k\in\N$ and $\lim_{k\to\infty} \mathscr I(\tau_k) = 0$. Since the rate function $\mathscr I$ is good,  the  set $\{\tau_k: k\in\N\}$ is weakly compact, and by passing to a subsequence we may assume that $\tau_k \to \tau$ weakly, as $k\to\infty$. By the lower semicontinuity and non-negativity of $\mathscr I$, this implies that $\mathscr I(\tau) = 0$. This means that $\tau = \mu_\infty^{(p)}$, which contradicts the assumption that $d(\tau_k,\mu_\infty^{(p)})\geq \eps$ for all $k\in\N$.

Hence, the infimum is negative and we have, for all sufficiently large $n> n_0(\eps)\in\N$ and some $c=c(\eps)>0$, 
$$
\Pro\left[d(\mu_n, \mu_\infty^{(p)}) \geq \eps \right] \leq e^{-c n^2}.
$$
From the Borel--Cantelli Lemma it follows that $d(\mu_n, \mu_\infty^{(p)})\to 0$ with probability $1$.

\section{Sketch of the proofs in the non self-adjoint case}\label{sec:ProofNonSelfAdjoint}

The proof of Theorem \ref{thm:MainNonSelfAdjoint} is essentially the same as the one of Theorem \ref{theo:main} (if $0<p<+\infty$) and Theorem \ref{theo:main_p_infty} if $p=+\infty$. For this reason, we only sketch the argument and highlight the points, where differences to the self-adjoint case occur. As above, we restrict our attention to the case $p<+\infty$.

One of the key ingredients in the self-adjoint case were the distributional representations \eqref{eq:schechtman_zinn_uniform} (for the uniform distribution) and \eqref{eq:schechtman_zinn_cone} (for the cone measure) taken from \cite{KPT2018a}. The proof of these representations in \cite{KPT2018a} was based on a change-of-variables formula (Proposition 4.1.1 in \cite{AGZ2010}). In the non self-adjoint case, a reformulation of this formula reads as follows (see \cite[Proposition 4.1.3]{AGZ2010}):

\begin{lemma}
Let $0<p<+\infty$ and $\beta\in\{1,2,4\}$. For each $n\in\N$ let $Z_n\in\SSS\B_p^n$ be uniformly distributed and, independently, let $\pi$ be a uniform random permutation on $\{1,\ldots,n\}$. Then the random vector $(s_{\pi(1)}^2(Z_n),\ldots,s_{\pi(n)}^2(Z_n))$ has density on $\R_+^n$ proportional to
$$
(x_1,\ldots,x_n)\mapsto\prod_{1\leq i<j\leq n}|x_i-x_j|^\beta\,\prod_{i=1}^nx_i^{{\beta\over 2}-1}\,\ind_{\{\sum_{i=1}^nx_i^{p/2}\leq 1\}}.
$$
\end{lemma}

Noting that the function
$$
\R_+^n\to\R,\quad(x_1,\ldots,x_n)\mapsto h_\beta(x_1,\ldots,x_n):=\prod_{1\leq i<j\leq n}|x_i-x_j|^\beta\,\prod_{i=1}^nx_i^{{\beta\over 2}-1}
$$
is homogeneous of degree
$$
m=m(n):={\beta n(n-1)\over 2}+n\Big({\beta\over 2}-1\Big)
$$
we can now argue as in the proof of Lemma 4.2 and Corollary 4.3 in \cite{KPT2018a} to conclude the following Schechtman-Zinn-type distributional representations for Schatten $p$-classes (note that instead of $p$ the value $p/2$ occurs at many places since we consider the squares of the singular values and not the singular values themself).

\begin{proposition}\label{prop:schechtman-zinn-non-self-adjoint}
Let $0<p<+\infty$ and $\beta\in\{1,2,4\}$. For each $n\in\N$ let $Z_n\in\SSS\B_p^n$ be uniformly distributed and, independently, let $\pi$ be a uniform random permutation of $\{1,\ldots,n\}$. Then
$$
(s_1^2(Z_n),\ldots,s_n^2(Z_n))\eqdistr U^{1\over n+m}{X_n\over\|X_n\|_{p/2}},
$$
where the random vector $X_n\in\R_+^n$ has joint density proportional to
\begin{align}\label{eq:JointDensitySchatten}
(x_1,\ldots,x_n)\mapsto e^{-n\sum_{i=1}^n|x_i|^{p/2}}\,h_\beta(x_1,\ldots,x_n)
\end{align}
and, independently of $X_n$, $U$ is uniformly distributed on $(0,1)$. Also, if $Z_n\in\partial\SSS\B_p^n$ is distributed according to the cone probability measure, one has that
$$
(s_1^2(Z_n),\ldots,s_n^2(Z_n))\eqdistr {X_n\over\|X_n\|_{p/2}}.
$$
\end{proposition}

In a next step, one proves an LDP for the sequence of empirical measures
$$
{1\over n}\sum_{i=1}^n\delta_{X_n^{(i)}},\qquad n\in\N
$$
determined by the coordinates $X_n^{(i)}$ of the random vector $X_n=(X_n^{(1)},\ldots,X_n^{(i)})$ with joint density \eqref{eq:JointDensitySchatten}. In fact, from \cite[Theorem 5.5.1]{hiai_petz} (with the choice $m(n)=n$, $\gamma=\beta/2$, $\alpha=1$ and $\gamma=0$ there) we conclude that this sequence satisfies an LDP on the space $\mathcal{M}(\R_+)$ of Borel probability measures on $\R_+$ endowed with the weak topology with speed $n^2$ and good rate function
$$
\mu\mapsto -{\beta\over 2}\int_{\R_+}\int_{\R_+}\log|x-y|\mu(\dint x)\mu(\dint y)+\int_{\R_+}|x|^{p/2}\mu(\dint x)+B.
$$
As in the proof for the self-adjoint case it is then possible to prove that the sequence of pairs
$$
\Big(\sum_{i=1}^n\delta_{X_n^{(i)}},\sum_{i=1}^n(X_n^{(i)})^{p/2}\Big),\qquad n\in\N,
$$
satisfies an LDP on the product space $\mathcal{M}(\R_+)\times\R_+$ supplied with the product of the weak topology on $\mathcal{M}(\R_+)$ and the standard topology on $\R_+$ with good rate function
$$
(\mu,m)\mapsto\begin{cases}
-{\beta\over 2}\int_{\R_+}\int_{\R_+}\log|x-y|\mu(\dint x)\mu(\dint y)+\int_{\R_+}|x|^{p/2}\mu(\dint x)+B &: m\geq \int_{\R_+}|x|^{p/2}\mu(\dint x)\\
+\infty &: m<\int_{\R_+}|x|^{p/2}\mu(\dint x).
\end{cases}
$$
Continuing the argument along the lines of the self-adjoint case one concludes that if $Z_n$ is uniformly distributed in $\SSS\B_p^n$ or distributed according to the cone measure on $\partial\SSS\B_p^n$ then the sequence of random probability measures
$$
{1\over n}\sum_{i=1}^n\delta_{n^{2/p}s_i^2(Z_n)},\qquad n\in\N,
$$
satisfies an LDP on $\mathcal{M}(\R_+)$ with speed $n^2$ and good rate function
\begin{align}\label{eq:RateFunctionNonSelfAdjoint}
\mu\mapsto\begin{cases}
-{\beta\over 2}\int_{\R_+}\int_{\R_+}\log|x-y|\mu(\dint x)\mu(\dint y)+C_{p,\beta} &: \int_{\R_+}|x|^{p/2}\mu(\dint x)\leq 1\\
+\infty &: \int_{\R_+}|x|^{p/2}\mu(\dint x)> 1,
\end{cases}
\end{align}
where $C_{p,\beta}$ is a constant depending on $p$ and on $\beta$.

To determine the value of $C_{p,\beta}$ we use the same symmetrization trick as in the proof of \cite[Proposition 6]{KPT2018b}. Namely, in order to find a Borel probability measure $\mu\in\mathcal{M}(\R_+)$ which minimizes the functional $\mu\mapsto-{\beta\over 2}\int_{\R_+}\int_{\R_+}\log|x-y|\mu(\dint x)\mu(\dint y)$ under the condition that $\int_{\R_+}|x|^{p/2}\mu(\dint x)\leq 1$, we maximize the expression ${\beta\over 2}\E\log|V-\widetilde{V}|$ over all random variables $V\geq 0$ under the condition that $\E V^{p/2}\leq 1$, where $\widetilde{V}$ is an independent copy of $V$. Let $\varepsilon$ be a Rademacher random variable taking the values $+1$ and $-1$ with probability $1/2$ and assume that $\varepsilon$ and $V$ are independent. Putting $U:=\varepsilon\sqrt{V}$ and denoting by $\widetilde{U}$ an independent copy of $U$ we have that
\begin{align*}
{\beta\over 2}\E\log|V-\widetilde{V}|={\beta\over 2}\E\log|U^2-\widetilde{U}^2|={\beta\over 2}\E\log|U-\widetilde{U}|+{\beta\over 2}\E\log|U+\widetilde{U}| = \beta\E\log|U-\widetilde{U}|.
\end{align*}
This means that we need to maximize the functional $\beta\E\log|U-\widetilde{U}|$ over all symmetric random variables $U$ on $\R$ under the condition that $\E|U|^p\leq 1$. It is known that the maximizer of this problem is given by a random variable with Lebesgue density $b_p^{-1}h_p(x/b_p)$, where $h_p(x)$ and $b_p$ are given by \eqref{eq:ullman_def}, see \cite{SaffBOOK}. Using now the computations carried out in \cite[Section 2.5]{KPT2018a} we conclude that the precise value of the constant $C_{p,\beta}$ in the above rate function \eqref{eq:RateFunctionNonSelfAdjoint} equals
$$
C_{p,\beta} = {\beta\over p}\log\left({\sqrt{\pi}p\Gamma({p\over 2})\over 2^p\sqrt{e}\Gamma({p+1\over 2})}\right).
$$
This completes the argument leading to \ref{thm:MainNonSelfAdjoint}. Finally, the proof of the law of large numbers in Corollary \ref{cor:SLLNSchatten} is literally the same as that of Corollary \ref{cor:LLN}.

\subsection*{Acknowledgement}

JP has been supported by a \textit{Visiting International Professor Fellowship} from the Ruhr University Bochum and its Research School PLUS. ZK and CT were supported by the DFG Scientific Network \textit{Cumulants, Concentration and Superconcentration}.

\bibliographystyle{plain}
\bibliography{sanov_LDP}

\end{document}